\documentclass[a4paper]{amsart}
\usepackage{amsfonts}
\usepackage{CJK}
\usepackage{amsthm}
\usepackage{amsmath}
\usepackage{amssymb}
\usepackage{graphicx}
\usepackage{layout}
\usepackage{mathrsfs}
\usepackage{epsfig}
\usepackage{indentfirst,latexsym,bm}
\usepackage{pifont}
\usepackage{txfonts}
\usepackage{color}
\usepackage[margin=1in]{geometry}
\begin{document}
\def\dis{\displaystyle}
\renewcommand{\theequation}{\thesection.\arabic{equation}}
\newtheorem{thm}{Theorem}[section]
\newtheorem{lemma}[thm]{Lemma}
\newtheorem{cor}[thm]{Corollary}
\newtheorem{pro}[thm]{Proposition}
\newtheorem{remark}[thm]{Remark}
\newtheorem{defin}[thm]{Definition}

\title[Suppression of blow-up in 3-D Keller-Segel]{Suppression of blow-up in 3-D Keller-Segel model via Couette flow in whole space}
\author{Shijin Deng}
\address{School of Mathematical Sciences, Shanghai Jiao Tong University, Shanghai 200240, P.R.China.}
\email{matdengs@sjtu.edu.cn}

\author{Binbin Shi}
\address{School of Mathematics and Statistics, Nanjing University of Science and Technology, Nanjing, 210094, P.R. China.}
\email{shibb@njust.edu.cn}

\author{Weike Wang}
\address{School of Mathematical Sciences, CMA-Shanghai and Institute of Natural Science,
Shanghai Jiao Tong University, Shanghai 200240, P.R.China.}
\email{wkwang@sjtu.edu.cn}
\date{}

\maketitle

\noindent{\bf Abstract.} In this paper, we study the 3-D parabolic-parabolic and parabolic-elliptic Keller-Segel models with Couette flow in $\mathbb{R}^3$. We prove that the blow-up phenomenon of solution can be suppressed by enhanced dissipation of large Couette flows. Here we develop Green's function method to describe the enhanced dissipation via a more precise space-time structure and obtain the global existence together with pointwise estimates of the solutions. The result of this paper shows that the enhanced dissipation exists for all frequencies in the case of whole space and it is reason that we obtain global existence for 3-D Keller-Segel models here. It is totally different from the case with the periodic spatial variable $x$ in \cite{bh,h}. This paper provides a new methodology to capture dissipation enhancement and also a surprising result which shows a totally new mechanism.

\vskip .05in



\vskip .05in


\section{Introduction}

In this paper, we are interested with the 3-D Keller-Segel model {{in}} the background of a large Couette flow:
\begin{equation}
\label{pro}
\begin{cases}
\partial_t n+ A y \partial_x n-\Delta n+\nabla\left(n\nabla c\right)=0,\\
\epsilon \Big(\partial_t c+Ay\partial_x c\Big)-\Delta c=n-c, \\
n(x,y,z,0)=n_0(x,y,z), \ \ (x,y,z)\in \mathbb{R}^3,\\
c(x,y,z,0)=c_0(x,y,z), \ \ (x,y,z)\in \mathbb{R}^3, \text{ for }\epsilon=1.
\end{cases}
\end{equation}
Here, the unknown functions $n(x,y,z,t)$ and $c(x,y,z,t)$ denote the micro-organism density and the chemo-attractant density respectively, the parameter $\epsilon\in\{0,1\}$ and $A$ is a large positive constant.

\vskip .05in

When $A\equiv 0$ and the parameter $\epsilon=0, 1$, it goes back to the classical parabolic-elliptic Keller-Segel model and parabolic-parabolic Keller-Segel model respectively. It is well-known that the solutions for Keller-Segel model in multi-D may blow up in finite time if the initial function {{$n_0$}} is large in $L^1$ norm. More precisely, for initial-boundary value problem in 2-D with a bounded domain, if the $L^1$ norm of the initial function {{$n_0$}} is less than $8\pi$, there exists a unique global solution; and if the $L^1$ norm of the initial function {{$n_0$}} exceeds $8\pi$, the solution may blow up in finite time. One could refer to \cite{hv,hv2,n,ss} for more details. In 3-D and higher dimensional cases, the solution blows up even for initial functions $n_0$ arbitrary small in $L^1$ norm, see \cite{sw,w}. For Cauchy problem, similar results could be found in \cite{kx} for periodic domain case and \cite{cc,s} for for whole space case.

\vskip .05in

The case $A\neq 0$ is corresponding to chemotactic processes in the background of a Couette flow. In fact, it is a more realistic scenario that chemotactic processes take place in a moving fluid, and the possible effects and related problems resulting from the interactions between the chemotactic process and the fluid transport have been widely investigated (see \cite{ccdl,ll,lo,w0} and the references therein).

\vskip .05in

The study of Keller-Segel model with an incompressible flow is also one of those attempts. An interesting question arising is whether one can suppress the possible finite time blow-up by the mixing effect coming from the fluid transport. For the parabolic-elliptic case, Kiselev and Xu \cite{kx} considered the suppression from the relaxation enhancing flow introduced in \cite{ckrz}, and proved that the solution of the advective Keller-Segel equation does not blow up in finite time provided that the amplitude of the relaxation enhancing flow is large enough; Bedrossian and He \cite{bh} found that shear flows have a different suppression effect in the sense that sufficiently large shear flows could prevent the blow-up in 2-D but could not guarantee the global existence in 3-D if the initial mass is greater than $8\pi$. For the parabolic-parabolic case, He \cite{h} proved that shear flows can also suppress the blow-up in 2-D under a certain smallness assumption on the derivative of initial function {{$c_0$}}. For related models, Zeng, Zhang and Zi \cite{zzz} considered the 2-D Keller-Segel-Navier-Stokes equations near the Couette flow and proved the global existence of solutions for both the parabolic-elliptic case and the parabolic-parabolic case if the amplitude of the Couette flow is large enough; Other related research can refer to \cite{fsw,h1}. The additional flows studied in those references are found to provide an enhanced dissipation effect from fluid, which could help the dissipation terms dominate even in the nonlinear level.

\vskip .05in

However, in most of the previous work, the $x$-variable is required to be in a periodic domain which ensures that the corresponding {{frequencies}} are discrete. In such a case, the situation is quite clear since the non-zero mode has a spectrum gap which helps us to capture the dissipation enhancement and for the zero mode there is completely no enhanced dissipation effect. If this requirement is removed and $x\in \mathbb{R}$ is considered, the enhanced dissipation effect in such a situation is unclear since the corresponding {{spectra}} are continuous and one could not separate zero and non-zero modes as before. Recently, Zelati and Gallay \cite{zg} considered the heat equation with a shear flow and $(x,y)\in\mathbb{R}\times [0,L]$ and proved that there is the enhanced dissipation effect and the Taylor dispersion for high and low frequency parts respectively.

\vskip .05in

In this paper, we consider the Keller-Segel model in the background of a Couette flow with spatial variables $(x,y,z)\in\mathbb{R}^3$. We obtain a detailed pointwise description of the solution by Green's function method to overcome the difficulty caused by $x\in\mathbb{R}$. It is interesting that we prove even in 3-D the blow-up of the solution for Pltlak-Keller-Segel model could be suppressed by a Couette flow with a large enough amplitude in the whole space case. It is totally different from the facts for periodic case studied in \cite{bh,h}: in 3-D periodic case, the effects from shear flows including Couette flow are not enough to suppress the blow-up of the solution. It seems that the continuity of the {{spectra}} also brings "certain" dissipation enhancement effects for the zero frequency part. In this paper, we put both parabolic-parabolic case and parabolic-elliptic case in the framework of Green's function method and our method and results also hold for the general case with spatial variables in $\mathbb{R}^n (n\ge 2)$; For simplicity and to show the intrinsic difference between the whole space case and the periodic case, we just consider the $\mathbb{R}^3$ case here.

\vskip .05in

The main results of this paper are as follows:
\begin{thm}\label{main}
(Parabolic-Parabolic case.) Suppose $\epsilon=1$ in \eqref{pro} and the initial functions $n_0(x,y,z), c_0(x,y,z)$ satisfy
\begin{equation}
\label{ini}
\left|n_0(x,y,z)\right|\le C_0e^{-(|x|+|y|+|z|)/C_*},
\end{equation}
and
\begin{equation}
\label{inic}
\left|c_0(x,y,z)\right|\le C_0e^{-(|x|+|y|+|z|)/C_*},\ \ \ \ \ \ \ \ \left|\nabla c_0(x,y,z)\right|\le C_0^*e^{-(|x|+|y|+|z|)/C_*},
\end{equation}
for constants $C_0, C_*>1$ and $0<C_0^*\ll 1$. Then there exists a positive constant $C_0\ll A< C_0/C_0^*$ such that the solution $(n(x,y,z,t),c(x,y,z,t))$ of \eqref{pro} exists globally in time and satisfies
\begin{equation}\nonumber
\left|n(x,y,z,t)\right|
\le CC_0 \mathscr{A}(t;A,\theta,\gamma)\mathscr{W}(x,y,z,t;C'_1,C''_1,C''_1;A),
\end{equation}
\begin{equation}\label{ccd}
\left|c(x,y,z,t)\right|\\
\le CC_0 A^{1/2-\gamma/2}\chi(t) \mathscr{A}(t;A,\theta,\gamma) \mathscr{W}(x,y,z,t;C'_1,C''_1,C''_1;A).
\end{equation}
Here, {{the time decay structure function $\mathscr{A}(t;A,\theta,\gamma)$ and the wave pattern function $ \mathscr{W}(x,y,z,t;C'_1,C''_1,C''_1;A)$}} are defined as follows
\begin{equation}\label{atime}
\mathscr{A}(t;A,\theta,\gamma)\equiv \begin{cases}
1,\ \ \ \ &\text{ for } \ 0< t\le A^{-\theta},\\
A^{-(1-\theta)\gamma}t^{-1/4-\gamma/4}\left(1+A^2t^2\right)^{-1/4+\gamma/4},\ \ \ \ &\text{ for } \ A^{-\theta}< t\le 1,\\
A^{-(1-\theta)\gamma}(1+t)^{-3/2}\left(1+A^2t^2\right)^{-1/2+\gamma/2},\ \ \ \ &\text{ for }\ t>1,
\end{cases}
\end{equation}
\begin{multline}\nonumber
\begin{aligned}
\mathscr{W}(x,y,z,t;C'_1,C''_1,C''_1;A)\equiv&\left(\exp{\left(\frac{-\left(x-\frac{At}2 y\right)^2}{C'_1t\left(1+A^2 t^2\right)}\right)}+\exp{\left(-\frac{\left|x-\frac{A t}2 y\right|}{C'_1(1+At)}\right)}\right)\\
&\cdot\left(\exp{\left(-\frac{y^2}{C''_1t}\right)}+\exp{\left(-\frac{|y|}{C''_1}\right)}\right)
\cdot\left(\exp{\left(-\frac{z^2}{C''_1t}\right)}+\exp{\left(-\frac{|z|}{C''_1}\right)}\right)
\end{aligned}
\end{multline}
and $\chi(t)$ is a cut-off function:
$$\chi(t)\equiv \begin{cases}
0,\ \ \ \ \text{ for }0< t\le 1\\
1,\ \ \ \ \text{ for }t>1.
\end{cases}$$
The positive constants $C, C'_1, C''_1$ are independent of $C_0$ and $A$, the parameters $\theta, \gamma$ satisfy $\theta\in(2/3,1)$ and $\gamma\in(1/3,1/2]$ respectively.
\end{thm}

In Theorem \ref{main}, for the Parabolic-Parabolic case, the initial data $n_0$ and $c_0$ could be large, while the derivative of initial function $c_0$ is required to be small enough:
$$|\nabla c_0(x,y,z)|<C_0/A\ll 1.$$
The decay structures of unknown functions $n(x,y,z,t)$ and $c(x,y,z,t)$ are also a little bit different: for long time scale, there is an extra large factor $A^{1/2-\gamma/2}\in [0,A^{1/3})$ in \eqref{ccd} for $c(x,y,z,t)$. For the Parabolic-Elliptic case, there are no such differences between $n(x,y,z,t)$ and $c(x,y,z,t)$:

\begin{thm}\label{main0}
(Parabolic-Elliptic case.) Suppose $\epsilon=0$ in \eqref{pro} and the initial function $n_0(x,y)$ satisfies \eqref{ini} for constants $C_0, C_*>1$. Then there exists a positive constant $A\gg C_0$ such that the solution of \eqref{pro} exists globally and satisfies
\begin{equation}\nonumber
\left|n(x,y,z,t)\right|, \left|c(x,y,z,t)\right|\\
\le CC_0 \mathscr{W}(x,y,z,t;C'_1,C''_1,C''_1;A)\mathscr{A}(t;A,\theta,\gamma)
\end{equation}
for positive constants $C, C'_1, C''_1$ independent of $C_0$ and $A$, $\theta\in(2/3,1), \gamma\in(0,1/2]$. Here, $\mathscr{A}(t;A,\theta,\gamma)$ and $ \mathscr{W}(x,y,z,t;C'_1,C''_1,C''_1;A)$ defined by \eqref{atime} and {{\eqref{wavest}}} respectively.
\end{thm}

\begin{remark} The main difference between the Parabolic-Parabolic case and the Parabolic-Elliptic case comes from the initial assumption \eqref{inic} in which the derivatives of initial function $c(x,y,z,0)$ are required to be small enough. This requirement should be an intrinsic one. Compared with $n(x,y,z,t)$, the derivative $\nabla c$ is involved in the nonlinear term which brings a destabilizing effect of the strong shear flow. Similar smallness assumption is also posed in \cite{h,zzz}.
\end{remark}

\begin{remark} For both cases, it is quite obvious that from the pointwise estimate of $n(x,y,z,t)$, the decay rate in $L^p(p\ge 2)$ norm is as follows:
$$\left\|n(\cdot,\cdot,\cdot,t)\right\|_{L^p_{x,y,z}}\le CC_0\cdot \begin{cases}
A^{(1-\theta)/p},\ \ \ \ &\text{ for }\ 0< t\le A^{-\theta},\\
A^{-(1-\theta)\gamma}t^{-1/4-\gamma/4}\left(1+A^2t^2\right)^{-1/4+\gamma/4+1/2p},\ \ \ \ &\text{ for }\ A^{-\theta}< t\le 1,\\
A^{-(1-\theta)\gamma}(1+t)^{-3/2(1-1/p)}\left(1+A^2t^2\right)^{-1/2+\gamma/2+1/2p},\ \ \ \ &\text{ for }\ t>1.
\end{cases}$$
Since $-1/2+\gamma/2+1/2p\le 0$ ensured by $\gamma\in(0,1/2]$, the long-time behavior of the solution (for $t>1$) has an enhanced dissipation factor $\left(1+A^2t^2\right)^{-1/2+\gamma/2+1/2p}$ in $L^p(p\ge 2)$ norm and also gains a small constant factor $A^{-(1-\theta)\gamma}$ since $A\gg 1$. However, when $t\le 1$, there may be an increasing process for the $L^p$ norm.
\end{remark}

In Theorem \ref{main}-\ref{main0}, we show that the enhanced dissipation effect of a Couette flow can suppress the blow up of solutions for the Keller-Segel model in $\mathbb{R}^3$, and establish the global existence of classical solutions. In fact, based on the local well-posedness and the regularity criterion (see Proposition 3.1 in Section 3), one only needs to obtain $L^\infty$ estimates to ensure the global existence of the solutions. In this paper, we fulfill the $L^\infty$ estimates by using Green's function method and it results in pointwise space-time structures of the solutions together with the global existence in Theorem \ref{main}-\ref{main0}. We should mention that the decaying structures in spatial variables required in initial conditions \eqref{ini}-\eqref{inic} are for technical reason, but it is interesting that one could show the equations \eqref{pro} with such initial data may blow up in finite time if $A=0$.

\vskip .05in

The use of Green's function in this paper is non-classical in the sense that the initial data are not required to be small enough and there is a variable coefficient in the linear parts of \eqref{pro}. In fact, this variable coefficient $Ay$ is the key ingredient for dissipation enhancement, which means that different from previous cases, one should not only overcome difficulties caused by variable coefficients, but also capture the advantage arising from them. The construction of Green's function for linearized equations with variable coefficients is highly non-trivial and most of the solved cases highly depend on the detailed structures of the variable coefficients. One could refer to \cite{l,lw,ly,lz,lz2,y} for the constructions of Green's functions for shock profiles and viscous rarefaction waves. Here, we use the precise formula given in \cite{mp} of Green's function for the heat equation in the background of a Couette flow in 2-D
$$\begin{cases}
\partial_t \mathbb{G}+Ay \partial_x \mathbb{G}=\partial_{xx}\mathbb{G}+\partial_{yy}\mathbb{G},\\
\mathbb{G}(x,y,0;y_0)=\delta(x,y-y_0),
\end{cases}$$
\begin{equation}
\label{greenp2d}
\mathbb{G}(x,y,t;y_0)=\frac{1}{4\pi t}\left(1+\frac{1}{12}A^2 t^2\right)^{-1/2}
\cdot \exp{\left(\frac{-\left(x-\frac{At}{2}(y+y_0)\right)^2-\left(1+\frac 1{12}A^2t^2\right)(y-y_0)^2}{4t\left(1+\frac{1}{12}A^2 t^2\right)}\right)},
\end{equation}
to construct an analogue for \eqref{pro} and the details could be found in Section 2. The treatment of the "large" initial data depends on an observation of the Green's function: compared with the classical heat kernel in 2D $\mathbb{H}(x,y,t)\equiv \frac 1{4\pi t}\exp\left(\frac{-x^2-y^2}{4t}\right)$, there is an extra factor $\left(1+\frac{1}{12}A^2 t^2\right)^{-1/2}$ in $\mathbb{G}(x,y,t;y_0)$ given by \eqref{greenp2d}. This factor is an evidence of dissipation enhancement in $L^p (p>1)$ norm and also suggests that in the linear level for long time the solution should become sufficiently small since
$$\left(1+\frac{1}{12}A^2 t^2\right)^{-1/2}\ll 1,\ \ \ \ \ \ \ \ \text{for }t\gg A^{-1}.$$
In the nonlinear level, the time evolution of the solution is more complicated due to the nonlinear effect: a longer time is needed for the solution to become small enough. Also an interesting point in the analysis for nonlinear wave interactions is the fact that the singularity $t^{-1/2}$ in time occurs together with the extra decaying factor $\left(1+\frac{1}{12}A^2 t^2\right)^{-1/2}$ and due to this fact non-integrable singularities in time variable appear in the nonlinear terms. We overcome this trouble based on another key observation that the balancing point of these two effects, the singularity in time and the extra decaying factor, is $t=O(A^{-2/3})$, and describe the solution in different time regions according to this balancing point (see \eqref{atime}). The pointwise estimate used in this paper is a powerful tool for a clear description of space-time structures of the solution and the detailed information obtained here helps us to capture the enhanced dissipation effects for all frequencies. It is a pity that the Green's functions for dissipative equations with other shear flows are still unclear and we hope that the method developed here can shed a light on other cases and they are left for the future.

\vskip .05in

Finally, we make a comparison about our results and those in \cite{bh}. Theorem \ref{main} and Theorem \ref{main0} show that for both the Parabolic-Parabolic case and the Parabolic-Elliptic case, a Couette flow with a sufficiently large amplitude can suppress the blow up of "large" solutions in $\mathbb{R}^3$. While in \cite{bh}, for periodic $x$ case in 3-D, a large Couette flow is not enough to guarantee the global existence of solutions and for global solutions, the initial mass should be less than $8\pi$ as well. The main reason is that there is no enhanced dissipation effect for the parts with zero frequency in $x$-direction and the equations for zero mode are similar to a classical 2-D Keller-Segel model. As show in our paper by the pointwise estimate method, dissipation enhancement occurs in all frequency for whole space case. Thus, in analysis we do not need to make different treatments for the zero-mode and non-zero modes. Compared with the periodic domain case, it is a surprising result and a totally new mechanism.

\vskip .05in

The rest of the paper is arranged as follows: In Section 2, we prepare all the tools needed including the Green's functions for $n(x,y,z,t)$ and $c(x,y,z,t)$ and also computational lemmas for nonlinear wave interactions and pointwise structure of $c(x,y,z,t)$. In Section 3, we first give the local well-posedness and the regularity criterion; then we pose ansatz assumptions based on initial wave propagations, and then apply the tools obtained in Section 2 and bootstrap for pointwise space-time structures of the solutions; the precise pointwise estimates for the solutions also give $L^p(2\le p\le \infty)$ bounds which results in the global existence of the solutions. Finally, we list all the wave pattern functions used in this paper and the interactions of different wave patterns in the appendix. Throughout the paper, $C$ denotes a generic positive constant which could change values from line to line.

\setcounter{equation}{0}

\section{Green's functions and computational lemmas}

In this section, we mainly prepare the analytical tools for pointwise estimates in Section 3. The tools include basic estimates for Green's functions and lemmas for linear and nonliear wave interactions. To make full use of the precise formula of Green's function for 2-D given in \cite{mp}, the equations in \eqref{pro} are considered separately. Thus in the construction of Green's function for $c(x,y,z,t)$, the linear term $n(x,y,z,t)$ will not be involved. It is treated as a linear "source" term which yields linear wave interactions

\subsection{Green's Functions}

We treat the two equations in \eqref{pro} separately in the construction of Green's function.

\subsubsection{Green's function for $n(x,y,z,t)$}

Define the Green's function $\mathbb{G}(x,y,z,t;y_0)$ to be the solution of
\begin{equation}
\label{green}
\begin{cases}
\partial_t \mathbb{G}+A y\partial_x \mathbb{G}-\Delta \mathbb{G}=0,\\
\mathbb{G}(x,y,z,0;y_0)=\delta_3(x,y-y_0,z),
\end{cases}
\end{equation}
with the 3-dimensional Dirac-delta function $\delta_3(x,y,z)$. It was shown in \cite{mp} that there is an explicit formula \eqref{greenp2d} for the solution of \eqref{green} in 2-D case and according to that formula, one could put down the solution $\mathbb{G}(x,y,z,t;y_0)$ as well
\begin{equation}
\label{greenp}
\mathbb{G}(x,y,z,t;y_0)=\frac{1}{\left(4\pi t\right)^{3/2}}\left(1+\frac{1}{12}A^2 t^2\right)^{-1/2}
\cdot \mathbb{W}(x,y,z,t;y_0;1,1,1;A),
\end{equation}
with the wave structure function $\mathbb{W}(x,y,z,t;y_0;D_1,D_2,D_3;A)$ defined as follows
\begin{equation}\nonumber
\mathbb{W}(x,y,z,t;y_0;D_1,D_2,D_3;A)\equiv \exp{\left(\frac{-\left(x-\frac{At}{2}(y+y_0)\right)^2}{4D_1t\left(1+\frac{1}{12}A^2 t^2\right)}-\frac{(y-y_0)^2}{4D_2t}-\frac{z^2}{4D_3t}\right)},
\end{equation}
for positive constants $D_1,D_2,D_3$.
A direct computation shows that extra decay rates for derivatives will be obtained and such effects are anisotropic:
\begin{multline}\label{greenpx}
\begin{aligned}
\partial_{x}\mathbb{G}(x,y,z,t;x_0,y_0)&=-\frac{x-x_0-\frac{At}{2}(y+y_0)}{2t\left(1+\frac{1}{12}A^2 t^2\right)}
\cdot \mathbb{W}(x,y,z,t;y_0;1,1,1;A)\\
&=O(1)t^{-2}\left(1+\frac{1}{12}A^2 t^2\right)^{-1}
\cdot \mathbb{W}(x,y,z,t;y_0;C_1,1,1;A),
\end{aligned}
\end{multline}

\begin{multline}\label{greenpy}
\begin{aligned}
\partial_{y_0}\mathbb{G}(x,y,z,t;x_0,y_0)
&=\left(\frac{At\left(x-x_0-\frac{At}{2}(y+y_0)\right)}{4t\left(1+\frac{1}{12}A^2 t^2\right)}+\frac{y-y_0}{2t}\right)
\cdot \mathbb{W}(x,y,z,t;y_0;1,1,1;A)\\
&=O(1)t^{-2}\left(1+\frac{1}{12}A^2 t^2\right)^{-1/2}
\cdot \mathbb{W}(x,y,z,t;y_0;C_1,C_1,1;A),
\end{aligned}
\end{multline}

\begin{equation}\label{greenpz}
\begin{aligned}
\partial_{z}\mathbb{G}(x,y,z,t;x_0,y_0)
&=-\frac{z}{2t}\cdot \mathbb{W}(x,y,z,t;y_0;1,1,1;A)\\
&=O(1)t^{-2}\left(1+\frac{1}{12}A^2 t^2\right)^{-1/2}
\cdot \mathbb{W}(x,y,z,t;y_0;1,1,C_1;A),
\end{aligned}
\end{equation}
for a positive constant $C_1>1$.

\vskip .05in

The following lemma shows the simple initial propagation for initial data decaying fast enough:

\begin{lemma}\label{lemma1}
For any $0\le \beta\le \frac {1}{2}$ and the positive constant $C_1$ in \eqref{greenpx}-\eqref{greenpz}, one has that
\begin{multline}\label{green1}
\begin{aligned}
&\left|\iiint_{\mathbb{R}^3}\mathbb{G}(x-x_0,y,z-z_0,t;y_0)\cdot e^{-\left(|x_0|+|y_0|+|z_0|\right)/C_*}dx_0dy_0dz_0\right|\\
=&O(1)t^{-1/2+\beta}\left(1+t\right)^{-1-\beta}\left(1+A^2t^2\right)^{-1/2+\beta}
\cdot\mathscr{W}^x(x,y,t;16,6C_*;A)\mathscr{W}^{o}(y,t;9,9C_*)\mathscr{W}^o(z,t;9,4C_*),
\end{aligned}
\end{multline}

\begin{multline}\label{green2}
\begin{aligned}
&\left|\iiint_{\mathbb{R}^3}\partial_{x_0}\mathbb{G}(x-x_0,y,z-z_0,t;y_0)\cdot e^{-\left(|x_0|+|y_0|+|z_0|\right)/C_*}dx_0dy_0dz_0\right|\\
=&O(1)t^{-1+\beta}\left(1+t\right)^{-1-\beta}\left(1+A^2t^2\right)^{-1+\beta}
\cdot\mathscr{W}^x(x,y,t;16C_1,6C_*;A)\mathscr{W}^{o}(y,t;9,9C_*)\mathscr{W}^o(z,t;9,4C_*),
\end{aligned}
\end{multline}

\begin{multline}\label{green3}
\begin{aligned}
&\left|\iiint_{\mathbb{R}^2}\partial_{y_0}\mathbb{G}(x-x_0,y,z-z_0,t;y_0)\cdot e^{-\left(|x_0|+|y_0|+|z_0|\right)/C_*}dx_0dy_0dz_0\right|,\\
=&O(1)t^{-1+\beta}\left(1+t\right)^{-1-\beta}\left(1+A^2t^2\right)^{-1/2+\beta}
\cdot\mathscr{W}^x(x,y,t;16C_1,6C_*;A)\mathscr{W}^{o}(y,t;9C_1,9C_*)\mathscr{W}^o(z,t;9,4C_*),
\end{aligned}
\end{multline}

\begin{multline}\label{green4}
\begin{aligned}
&\left|\iiint_{\mathbb{R}^2}\partial_{z_0}\mathbb{G}(x-x_0,y,z-z_0,t;y_0)\cdot e^{-\left(|x_0|+|y_0|+|z_0|\right)/C_*}dx_0dy_0dz_0\right|\\
=&O(1)t^{-1+\beta}\left(1+t\right)^{-1-\beta}\left(1+A^2t^2\right)^{-1/2+\beta}
\cdot\mathscr{W}^x(x,y,t;16,6C_*;A)\mathscr{W}^{o}(y,t;9,9C_*)\mathscr{W}^o(z,t;9C_1,4C_*).
\end{aligned}
\end{multline}
Here, we denote the wave structure functions $\mathscr{W}^x(x,y,t;D_1,D_2;A)$ and $\mathscr{W}^o(y,t;D_1,D_2)$ with positive constants $D_1, D_2$ as follows:
\begin{equation}\nonumber
\mathscr{W}^x(x,y,t;D_1,D_2;A)\equiv \exp{\left(\frac{-\left(x-\frac{At}2 y\right)^2}{D_1t\left(1+A^2 t^2\right)}\right)}+\exp{\left(-\frac{\left|x-\frac{A t}2 y\right|}{D_2(1+At)}\right)},
\end{equation}
\begin{equation}\nonumber
\mathscr{W}^o(y,t;D_1,D_2)\equiv \exp{\left(-\frac{y^2}{D_1t}\right)+\exp{\left(-\frac{|y|}{D_2}\right)}}.
\end{equation}
\end{lemma}

\begin{remark}This lemma reveals that the extra decaying rate $\left(1+A^2t^2\right)^{-1/2}$ goes with the singularity of heat kernel $t^{-1/2}$ and this extra decaying rate disappears if one removes the singularity $t^{-1/2}$ as what one does in dealing with the heat equation. Here and in the rest of the paper, we use a parameter $\beta\in(0,1/2)$ to balance these two effects.
\end{remark}

\begin{proof}
When $t\ge 1$, one has that $\frac{1}{2t}\le\frac{1}{1+t}$ and thus
\begin{multline}\label{e2}
\begin{aligned}
&\left|\iiint_{\mathbb{R}^3}\mathbb{G}(x-x_0,y,z-z_0,t;y_0)\cdot e^{-\left(|x_0|+|y_0|+|z_0|\right)/C_*}dx_0dy_0dz_0\right| \\
\leq &C\left(1+t\right)^{-3/2}\left(1+A^2t^2\right)^{-1/2}\\
&\cdot \left(\iint_{\mathbb{R}^2}\exp{\left(\frac{-\left(y-y_0\right)^2-\left(z-z_0\right)^2}{4t}-\frac{|y_0|+|z_0|}{C_*}\right)}
\cdot \left(\int_{\mathbb{R}}\exp{\left(-\frac{\left(x-x_0-\frac{At}2(y+y_0)\right)^2}{4t\left(1+\frac{1}{12}A^2 t^2\right)}-\frac{|x_0|}{C_*}\right)}dx_0\right)dy_0dz_0\right)\\
\leq& C\left(1+t\right)^{-3/2}\left(1+A^2t^2\right)^{-1/2}\cdot\left(\left(\int_{|z|\ge 3|z_0|}+\int_{|z|\le 3|z_0|}\right)\exp{\left(-\frac{\left(z-z_0\right)^2}{4t}-\frac{|z_0|}{C_*}\right)}dz_0\right)\\
&\cdot\left(\int_{\mathbb R}\exp{\left(-\frac{\left(y-y_0\right)^2}{4t}-\frac{|y_0|}{C_*}\right)}
\cdot \left(\left(\int_{\left|x-\frac{At}2( y+y_0)\right|\ge 3|x_0|}+\int_{\left|x-\frac{At}2( y+y_0)\right|\le 3|x_0|}\right)\exp{\left(-\frac{\left(x-x_0-\frac{At}2(y+y_0)\right)^2}{4t\left(1+\frac{1}{12}A^2 t^2\right)}-\frac{|x_0|}{C_*}\right)}dx_0\right)dy_0\right)\\
\leq & C\left(1+t\right)^{-3/2}\left(1+A^2t^2\right)^{-1/2}\cdot\mathscr{W}^o(z,t;9,4C_*)\\
&\cdot \left(\left(\int_{\left\{\left|x-\frac{At}2 y\right|\ge 2At|y_0|\right\}\cap\left\{\left|y\right|\ge 3\left|y_0\right|\right\}}+\int_{\left\{\left|x-\frac{At}2 y\right|\ge 2At|y_0|\right\}\cap\left\{\left|y\right|\le 3\left|y_0\right|\right\}}+\int_{\left\{\left|x-\frac{At}2 y\right|\le 2At|y_0|\right\}\cap\left\{\left|y\right|\ge 3\left|y_0\right|\right\}}+\int_{\left\{\left|x-\frac{At}2 y\right|\le 2At|y_0|\right\}\cap\left\{\left|y\right|\le 3\left|y_0\right|\right\}}\right)\right.\\
&\left.\ \exp{\left(-\frac{\left(y-y_0\right)^2}{4t}-\frac{|y_0|}{C_*}\right)}
\cdot\left(\exp{\left(-\frac{\left(x-\frac{At}2(y+y_0)\right)^2}{9t\left(1+\frac{1}{12}A^2 t^2\right)}\right)}+\exp\left(-\frac{\left|x-\frac{At}2( y+y_0)\right|}{4C_*}\right)\right)dy_0\right)\\
\leq &C\left(1+t\right)^{-3/2}\left(1+A^2t^2\right)^{-1/2}\cdot \mathscr{W}^x(x,y,t;16,9C_*;A)\mathscr{W}^{o}(y,t;9,9C_*)\mathscr{W}^o(z,t;9,4C_*),
\end{aligned}
\end{multline}
which verifies \eqref{green1} for $t\ge 1$. Here, to maximum the decaying rate in time $t$, we only integrate the exponential functions, e.g.
\begin{multline}\label{e1}
\begin{aligned}
&\left(\int_{\left|x-\frac{At}2( y+y_0)\right|\ge 3|x_0|}+\int_{\left|x-\frac{At}2( y+y_0)\right|\le 3|x_0|}\right)\exp{\left(-\frac{\left(x-x_0-\frac{At}2(y+y_0)\right)^2}{4t\left(1+\frac{1}{12}A^2 t^2\right)}-\frac{|x_0|}{C_*}\right)}dx_0\\
=&\left(\int_{\left|x-\frac{At}2( y+y_0)\right|\ge 3|x_0|}+\int_{\left|x-\frac{At}2( y+y_0)\right|\le 3|x_0|}\right)\exp{\left(-\frac{\left(x-x_0-\frac{At}2(y+y_0)\right)^2}{4t\left(1+\frac{1}{12}A^2 t^2\right)}-\frac{3|x_0|}{4C_*}-\frac{|x_0|}{4C_*}\right)}dx_0\\
\le& \left(\exp{\left(-\frac{\left(x-\frac{At}2(y+y_0)\right)^2}{9t\left(1+\frac{1}{12}A^2 t^2\right)}\right)}+\exp\left(-\frac{\left|x-\frac{At}2( y+y_0)\right|}{4C_*}\right)\right)
\int_{\mathbb{R}} e^{-|x_0|/4C_*}dx_0\\
\le &C\left(\exp{\left(-\frac{\left(x-\frac{At}2(y+y_0)\right)^2}{9t\left(1+\frac{1}{12}A^2 t^2\right)}\right)}+\exp\left(-\frac{\left|x-\frac{At}2( y+y_0)\right|}{4C_*}\right)\right).
\end{aligned}
\end{multline}
For $0<t<1$, one could integrate the heat kernels and an analogue of \eqref{e1} could be as follows:
\begin{multline}\label{2.10}
\begin{aligned}
&\left(\int_{\left|x-\frac{At}2( y+y_0)\right|\ge 4|x_0|}+\int_{\left|x-\frac{At}2( y+y_0)\right|\le 4|x_0|}\right)\exp{\left(-\frac{\left(x-x_0-\frac{At}2(y+y_0)\right)^2}{4t\left(1+\frac{1}{12}A^2 t^2\right)}-\frac{|x_0|}{C_*}\right)}dx_0\\
\le& \left(\exp{\left(-\frac{\left(x-\frac{At}2(y+y_0)\right)^2}{8t\left(1+\frac{1}{12}A^2 t^2\right)}\right)}+\exp\left(-\frac{\left|x-\frac{At}2( y+y_0)\right|}{4C_*}\right)\right)
\int_{\mathbb{R}} \exp{\left(\frac{-\left(x-x_0-\frac{At}2(y+y_0)\right)^2}{36t\left(1+\frac{1}{12}A^2 t^2\right)}\right)}dx_0\\
\le& Ct^{1/2}\left(1+A^2t^2\right)^{1/2}\left(\exp{\left(-\frac{\left(x-\frac{At}2(y+y_0)\right)^2}{8t\left(1+\frac{1}{12}A^2 t^2\right)}\right)}+\exp\left(-\frac{\left|x-\frac{At}2( y+y_0)\right|}{4C_*}\right)\right).
\end{aligned}
\end{multline}
Moreover, in the integration with respect to $y_0$ and $z_0$ in \eqref{e2}, we could compute similarly and also integrate the heat kernels instead of exponential functions to obtain an extra factor $t$. This verifies \eqref{green1} for the case $0<t<1$ and $\beta=1/2$. If one only integrates the heat kernel with respect to $y_0$ and $z_0$ and still integrates the exponential functions with respect to $x_0$, one could verify \eqref{green1} for $0<t<1$ and $\beta=0$. The estimates \eqref{green1} for $0<\beta<1/2$ could be obtained by interpolation.

Similarly, one could use \eqref{greenpx}-\eqref{greenpz} to verify the estimates \eqref{green2} and \eqref{green3} for the integrals involving derivatives.
\end{proof}

\subsubsection{Green's function for $c(x,y,z,t)$}

Similar to \eqref{green}, one could define the Green's function $\mathbb{G}_{c,1}(x,y,z,t;y_0)$ by the following Cauchy problem
\begin{equation}\label{greenc}
\begin{cases}
\partial_t \mathbb{G}_{c,1}+Ay \partial_x \mathbb{G}_{c,1}-\Delta \mathbb{G}_{c,1}=- \mathbb{G}_{c,1},\\
\mathbb{G}_{c,1}(x,y,z,0;y_0)=\delta_3(x,y-y_0,z))
\end{cases}
\end{equation}
for the parabolic-parabolic case, i.e. $\epsilon\equiv1$ in \eqref{pro} and the Green's function $\mathbb{G}_{c,0}(x,y,z)$ by the elliptic equation
\begin{equation}\label{greenc0}
-\Delta \mathbb{G}_{c,0}+\mathbb{G}_{c,0}=\delta_3(x,y,z)
\end{equation}
for the parabolic-elliptic case, i.e. $\epsilon\equiv0$ in \eqref{pro}.

Both problems \eqref{greenc} and \eqref{greenc0} could be solved explictly: $\mathbb{G}_{c,1}(x,y,z,t;y_0)$ is given by
\begin{equation}\label{gc}
\mathbb{G}_{c,1}(x,y,z,t;y_0)=e^{-t}\mathbb{G}(x,y,z,t;y_0)
\end{equation}
with $\mathbb{G}$ defined by \eqref{green} and thus one could refer to Lemma \ref{lemma1} for the estimates for $\mathbb{G}_{c,1}$ and its derivatives; $\mathbb{G}_{c,0}$ is the famous Yukawa potential in 3-D:
\begin{equation}
\label{yukawa}
\mathbb{G}_{c,0}(x,y,z)=-\frac{1}{4\pi {r}} e^{-{r}},\ \ \ \ \ \ \ \
\left|\nabla \mathbb{G}_{c,0}(x,y,z)\right|\le \frac{C(1+r)e^{-r}}{r^2},
\end{equation}
with ${r}\equiv \sqrt{x^2+y^2+z^2}$.

The above two Green's functions are associated with $c(x,y,z,t)$ in the following sense: for $\epsilon=1$ in \eqref{pro}, the solution $c(x,y,z,t)$ of \eqref{pro} could be represented as follows
\begin{multline}\label{c1}
\begin{aligned}
c(x,y,z,t)=&\iiint_{\mathbb{R}^3} \mathbb{G}_{c,1}(x-x_0,y,z-z_0,t;y_0) c(x_0,y_0,z_0,0)dx_0dy_0dz_0\\
&+\int_0^t\iiint_{\mathbb{R}^3} \mathbb{G}_{c,1}(x-x_0,y,z-z_0,t-\sigma;y_0) n(x_0,y_0,z_0,\sigma)dx_0dy_0dz_0d\sigma,
\end{aligned}
\end{multline}
and for the case $\epsilon=0$, one has that
\begin{equation}\label{c0}
c(x,y,z,t)=\iiint_{\mathbb{R}^3} \mathbb{G}_{c,0}(x-x_0,y-y_0,z-z_0) n(x_0,y_0,z_0,t)dx_0dy_0dz_0.
\end{equation}
Later, these two representations \eqref{c1} and \eqref{c0} will be used to obtain the pointwise space-time structures of $c(x,y,z,t)$ once one has the pointwise estimate of $n(x,y,z,t)$.

\subsection{Computational Lemmas: Nonlinear Wave Interactions}

The next three lemmas are prepared for nonlinear wave interactions: In the nonlinear level, the time evolution of the solutions is more complicated than the one in the linear level and one needs to balance the singularity and the enhanced dissipation effect revealed by Lemma \ref{lemma1}. Therefore the nonlinear solutions are described in three different time regions (the details could be found in \eqref{ansatz} in Section 3), and one should verify the ansatz assumptions in those time regions correspondingly. The whole process contains two steps: the interactions between different wave pattern functions are listed in the appendix and in the following we will prove three lemmas for time evolutions after wave interactions. Denote
\begin{equation}\nonumber
\mathscr{W}(x,y,z,t;D_1,D_2,D_3;A)\equiv \mathscr{W}^x(x,y,t;D_1,D_1;A)\cdot \mathscr{W}^o(y,t;D_2,D_2)\cdot \mathscr{W}^o(z,t;D_3,D_3),
\end{equation}
with positive constants $D_1, D_2, D_3$ and wave pattern functions $\mathscr{W}^x(x,y,t;D_1,D_2;A), \mathscr{W}^0(y,t;D_1,D_2)$ defined by \eqref{wavexx} and \eqref{waveyz} respectively.

\begin{lemma}\label{lemma2} For positive constants $\gamma\in(0,1), \theta\in(2/3,1), C_1>1$ and $C',C''$ satisfying the requirements of Lemma \ref{lemmaz}-\ref{lemma0}, there exists a constant $\epsilon_0\equiv (1+\gamma)(3\theta/2-1)/2>0$ such that
\begin{multline}\label{nn1}
\begin{aligned}
&\int_0^{\min\left\{A^{-\theta}, t\right\}} \iiint_{\mathbb{R}^3}(t-\sigma)^{-2}\left(1+\frac{1}{12}A^2 (t-\sigma)^2\right)^{-1/2}\\[2mm]
&\cdot \mathbb{W}(x-x_0,y,z-z_0,t-\sigma;y_0;C_1,C_1,C_1;A)
\cdot \mathscr{W}(x_0,y_0,z_0,\sigma;C',C'',C'';A)dx_0dy_0dz_0d\sigma\\
=&O(1)\mathscr{A}_1(t)\mathscr{W}(x,y,z,t;\frac 32C',\frac 32C'',\frac 32C'';A),
\end{aligned}
\end{multline}
with $\mathscr{A}_1(t)$ defined by
\begin{equation}\label{aa1}
\mathscr{A}_1(t)\equiv \begin{cases}
t^{1/2},\ \ \ \ &\text{ for }\ 0< t\le A^{-\theta},\\
A^{-\epsilon_0}A^{-(1-\theta)\gamma}t^{-1/4-\gamma/4}\left(1+A^2t^2\right)^{-1/4+\gamma/4},\ \ \ \ &\text{ for }\ A^{-\theta}< t\le 1,\\
A^{1-2\theta}(1+t)^{-2}\left(1+A^2t^2\right)^{-1/2},\ \ \ \ &\text{ for } \ t>1.
\end{cases}
\end{equation}


\end{lemma}

\begin{proof}
For any $t>0$, by direct computations based on Lemma \ref{lemmaz}-\ref{lemma0}, one has that
\begin{multline}\label{2.22}
\begin{aligned}
&\int_0^{t} \iiint_{\mathbb{R}^3}(t-\sigma)^{-2}\left(1+\frac{1}{12}A^2 (t-\sigma)^2\right)^{-1/2}\\[2mm]
&\cdot \mathbb{W}(x-x_0,y,z-z_0,t-\sigma;y_0;C_1,C_1,C_1;A)
\cdot \mathscr{W}(x_0,y_0,z_0,\sigma;C',C'',C'';A)dx_0dy_0dz_0d\sigma\\[2mm]
\le& C\left(\int_0^t (t-\sigma)^{-2}\left(1+\frac{1}{12}A^2 (t-\sigma)^2\right)^{-1/2}\cdot (t-\sigma)^{3/2}\left(1+A^2 (t-\sigma)^2\right)^{1/2}d\sigma\right)
\cdot \mathscr{W}(x,y,z,t;\frac 32C',\frac 32C'',\frac 32C'';A)\\
\le &Ct^{1/2}\mathscr{W}(x,y,z,t;\frac 32C',\frac 32C'',\frac 32C'';A).
\end{aligned}
\end{multline}
It verifies \eqref{nn1} for $t\in(0,A^{-\theta}]$.

If $t\in(A^{-\theta},2A^{-\theta}]$, there exists some constant $C>1$ such that
\begin{equation}\label{add}
t^{1/2}\le CA^{-\theta/2},\ \ \ \ \ \ \ \ t^{-1/4-\gamma/4}\ge A^{\theta(1+\gamma)/4}/C,\ \ \ \ \ \ \ \ \left(1+A^2t^2\right)^{-1/4+\gamma/4}\ge A^{-(1-\gamma)(1-\theta)/2}/C
\end{equation}
and thus combining \eqref{2.22} and \eqref{add} one has that
\begin{multline}\label{2.222}
\begin{aligned}
&\int_0^{A^{-\theta}} \iiint_{\mathbb{R}^3}(t-\sigma)^{-2}\left(1+\frac{1}{12}A^2 (t-\sigma)^2\right)^{-1/2}\\[2mm]
&\cdot \mathbb{W}(x-x_0,y,z-z_0,t-\sigma;y_0;C_1,C_1,C_1;A)
\cdot \mathscr{W}(x_0,y_0,z_0,\sigma;C',C'',C'';A)dx_0dy_0dz_0d\sigma\\[2mm]
\le &C t^{-1/4-\gamma/4}\left(1+A^2 t^2\right)^{-1/4+\gamma/4}\mathscr{W}(x,y,z,t;\frac 32C',\frac 32C'',\frac 32C'';A)\cdot
\Big(A^{-\theta/2} \cdot A^{-\theta(1+\gamma)/4+(1-\theta)(1-\gamma)/2} \Big)\\
\le &C A^{-\epsilon_0-\gamma(1-\theta)} t^{-1/4-\gamma/4}\left(1+A^2 t^2\right)^{-1/4+\gamma/4}\cdot\mathscr{W}(x,y,z,t;\frac 32C',\frac 32C'',\frac 32C'';A),
\end{aligned}
\end{multline}
since $-\theta/2-\theta(1+\gamma)/4+(1-\theta)(1-\gamma)/2= -\epsilon_0-\gamma(1-\theta)-\theta/2<-\epsilon_0-\gamma(1-\theta)$ with $\epsilon_0\equiv (1+\gamma)(3\theta/2-1)/2$. For $t\in (2A^{-\theta},1]$, $t-\sigma \ge t/2$ and $1+\frac 1{12}A^2(t-\sigma)^2 \ge \frac{1}{48}\left(1+A^2t^2\right)$ for $\sigma\in[0,A^{-\theta}]$ and thus one could combine Lemma \ref{lemmaz}-\ref{lemma0} to yield that
\begin{multline}\label{2.221}
\begin{aligned}
&\int_0^{A^{-\theta}} \iiint_{\mathbb{R}^3}(t-\sigma)^{-2}\left(1+\frac{1}{12}A^2 (t-\sigma)^2\right)^{-1/2}\\[2mm]
&\cdot \mathbb{W}(x-x_0,y,z-z_0,t-\sigma;y_0;C_1,C_1,C_1;A)
\cdot \mathscr{W}(x_0,y_0,z_0,\sigma;C',C'',C'';A)dx_0dy_0dz_0d\sigma\\
\le& C\left(1+A^2 t^2\right)^{-1/2}\cdot\mathscr{W}(x,y,z,t;\frac 32C',\frac 32C'',\frac 32C'';A)\cdot \int_0^{A^{-\theta}} (t-\sigma)^{-2}\cdot(t-\sigma) \left(\sigma^{1/2}\left(1+A^2\sigma^2\right)^{1/2}+1+A\sigma\right)d\sigma\\
\le &C A^{1-2\theta}t^{-1}\left(1+A^2t^2\right)^{-1/2}\cdot\mathscr{W}(x,y,z,t;\frac 32C',\frac 32C'',\frac 32C'';A)\\
\le& C t^{-1/4-\gamma/4}\left(1+A^2 t^2\right)^{-1/4+\gamma/4}\cdot \mathscr{W}(x,y,z,t;\frac 32C',\frac 32C'',\frac 32C'';A)
\cdot A^{1-2\theta} A^{\theta(3-\gamma)/4-(1+\gamma)(1-\theta)/2}\\
\le& CA^{-\epsilon_0-(1-\theta)\gamma}t^{-1/4-\gamma/4} \left(1+A^2t^2\right)^{-1/4+\gamma/4}
\cdot \mathscr{W}(x,y,z,t;\frac 32C',\frac 32C'',\frac 32C'';A).
\end{aligned}
\end{multline}
Then, one could combine \eqref{2.222} and \eqref{2.221} to verify \eqref{nn1} for $t\in(A^{-\theta},1]$.

When $t>1$ and $\sigma\le A^{-\theta}\ll 1$, it is obvious that $t-\sigma>t/2>(1+t)/4$ and thus similar to \eqref{2.221} one has that
\begin{multline}\label{2.36}
\begin{aligned}
&\int_0^{A^{-\theta}} \iiint_{\mathbb{R}^3}(t-\sigma)^{-2}\left(1+\frac{1}{12}A^2 (t-\sigma)^2\right)^{-1/2}\\[2mm]
&\cdot \mathbb{W}(x-x_0,y,z-z_0,t-\sigma;y_0;C_1,C_1,C_1;A)
\cdot \mathscr{W}(x_0,y_0,z_0,\sigma;C',C'',C'';A)dx_0dy_0dz_0d\sigma\\
\le& C(1+t)^{-2}\left(1+A^2 t^2\right)^{-1/2}\cdot \mathscr{W}(x,y,z,t;\frac 32C',\frac 32C'',\frac 32C'';A)\cdot\left(\int_0^{A^{-\theta}} \left(\sigma\left(1+A^2\sigma^2\right)^{1/2}+1+A\sigma\right)\left(\sigma^{1/2}+1\right)d\sigma\right)\\
\le& CA^{1-2\theta}(1+t)^{-2}\left(1+A^2 t^2\right)^{-1/2}\mathscr{W}(x,y,z,t;\frac 32C',\frac 32C'',\frac 32C'';A)
\end{aligned}
\end{multline}
which verifies \eqref{nn1} for $t>1$.
\end{proof}

\begin{lemma}\label{lemma3} For positive constants $\theta, \gamma\in(0,1), C_1>1$ and $C',C''$ satisfying the requirements of Lemma \ref{lemmaz}-\ref{lemma0}, it holds that
\begin{multline}\label{nn2}
\begin{aligned}
&\int_{\min\left\{A^{-\theta},t\right\}}^{\min\left\{1,t\right\}} \iiint_{\mathbb{R}^3}(t-\sigma)^{-2}\left(1+\frac{1}{12}A^2 (t-\sigma)^2\right)^{-1/2}\sigma^{-1/2-\gamma/2}
\left(1+A^2\sigma^2\right)^{-1/2+\gamma/2}
\\[2mm]
&\cdot \mathbb{W}(x-x_0,y,z-z_0,t-\sigma;y_0;C_1,C_1,C_1;A)
\cdot \mathscr{W}(x_0,y_0,z_0,\sigma;C',C'',C'';A)dx_0dy_0dz_0d\sigma\\
=&O(1)\mathscr{A}_2(t)\mathscr{W}(x,y,z,t;\frac 32C',\frac 32C'',\frac 32C'';A).
\end{aligned}
\end{multline}
Here, the function $\mathscr{A}_2(t)$ is defined to be
\begin{equation}\label{aa2}
\mathscr{A}_2(t)\equiv \begin{cases}
0,\ \ \ \ &\text{ for }\ 0< t\le A^{-\theta},\\
t^{-1/4-\gamma/4}\left(1+A^2t^2\right)^{-1/4+\gamma/4},\ \ \ \ &\text{ for }\ A^{-\theta}< t\le 1,\\
A^{\gamma}(1+t)^{-2}\left(1+A^2t^2\right)^{-1/2},\ \ \ \ &\text{ for }\ t>1.
\end{cases}
\end{equation}
\end{lemma}

\begin{proof}
It is obvious that \eqref{nn2} holds for $t\in(0,A^{-\theta}]$ since the upper and lower limits of the integral on the LHS of \eqref{nn2} are the same and the integral equals 0.

\vskip .05in

If $t\in(A^{-\theta},1]$, one has that $0<\sigma\le t\le 1$ and it combined with Lemma \ref{lemmaz}-\ref{lemma0} yields that
\begin{multline}\nonumber
\begin{aligned}
&\iiint_{\mathbb{R}^3}\mathbb{W}(x-x_0,y,z-z_0,t-\sigma;y_0;C_1,C_1,C_1;A)\cdot \mathscr{W}(x_0,y_0,z_0,\sigma;C',C'',C'';A)dx_0dy_0dz_0\\
\le& C(t-\sigma)^{3/2}\left(1+A^2(t-\sigma)^2\right)^{1/2}\mathscr{W}(x,y,z,t;\frac 32C',\frac 32C'',\frac 32C'';A),
\end{aligned}
\end{multline}
and
\begin{multline}\nonumber
\begin{aligned}
&\iiint_{\mathbb{R}^3}\mathbb{W}(x-x_0,y,z-z_0,t-\sigma;y_0;C_1,C_1,C_1;A)\cdot \mathscr{W}(x_0,y_0,z_0,\sigma;C',C'',C'';A)dx_0dy_0dz_0\\
\le& C(t-\sigma)\left(\sigma^{1/2}\left(1+A^2\sigma^2\right)^{1/2}+
\left(1+A\sigma\right)\right)\mathscr{W}(x,y,z,t;\frac 32C',\frac 32C'',\frac 32C'';A)\\
\le& C(t-\sigma)\left(1+A^2\sigma^2\right)^{1/2}\mathscr{W}(x,y,z,t;\frac 32C',\frac 32C'',\frac 32C'';A),
\end{aligned}
\end{multline}
and thus
\begin{multline}\label{above}
\begin{aligned}
&\iiint_{\mathbb{R}^3}\mathbb{W}(x-x_0,y,z-z_0,t-\sigma;y_0;C_1,C_1,C_1;A)\cdot \mathscr{W}(x_0,y_0,z_0,\sigma;C',C'',C'';A)dx_0dy_0dz_0\\
\le &C(t-\sigma)\mathscr{W}(x,y,z,t;\frac 32C',\frac 32C'',\frac 32C'';A)\cdot
\Big((t-\sigma)^{1/2\cdot(1+\gamma)/2}\left(1+A^2(t-\sigma)\right)^{1/2\cdot(1+\gamma)/2}\cdot\left(1+A^2\sigma^2\right)^{1/2\cdot (1-\gamma)/2}\Big).
\end{aligned}
\end{multline}
The inequality \eqref{above} results in that
\begin{multline}\nonumber
\begin{aligned}
&\int_{A^{-\theta}}^t \iiint_{\mathbb{R}^3}(t-\sigma)^{-2}\left(1+\frac{1}{12}A^2 (t-\sigma)^2\right)^{-1/2}\sigma^{-1/2-\gamma/2}
\left(1+A^2\sigma^2\right)^{-1/2+\gamma/2}
\\[2mm]
&\cdot \mathbb{W}(x-x_0,y,z-z_0,t-\sigma;y_0;C_1,C_1,C_1;A)\cdot \mathscr{W}(x_0,y_0,z_0,\sigma;C',C'',C'';A)dx_0dy_0dz_0d\sigma\\
\le& C\mathscr{W}(x,y,z,t;\frac 32C',\frac 32C'',\frac 32C'';A)\\
&\cdot\left(\left(\int_{A^{-\theta}}^{\max\{A^{-\theta},t/2\}}+\int_{\max\{A^{-\theta},t/2\}}^t\right) (t-\sigma)^{-3/4+\gamma/4}\left(1+A^2 (t-\sigma)^2\right)^{-1/4+\gamma/4}\sigma^{-1/2-\gamma/2}
\left(1+A^2\sigma^2\right)^{-1/4+\gamma/4}d\sigma\right)\\
\le& C t^{-1/4-\gamma/4}\left(1+A^2t^2\right)^{-1/4+\gamma/4}
\mathscr{W}(x,y,z,t;\frac 32C',\frac 32C'',\frac 32C'';A)
\end{aligned}
\end{multline}
and verifies \eqref{nn2} for $t\in(A^{-\theta},1]$.

\vskip .05in

If $t>2$ and thus $t-\sigma>t/2>(1+t)/4$ for $\sigma\le 1$, one has that
\begin{multline}\label{2.37}
\begin{aligned}
&\int_{A^{-\theta}}^1 \iiint_{\mathbb{R}^3}(t-\sigma)^{-2}\left(1+\frac{1}{12}A^2 (t-\sigma)^2\right)^{-1/2}\sigma^{-1/2-\gamma/2}
\left(1+A^2\sigma^2\right)^{-1/2+\gamma/2}\\[2mm]
&\cdot \mathbb{W}(x-x_0,y,z-z_0,t-\sigma;y_0;C_1,C_1,C_1;A)
\cdot \mathscr{W}(x_0,y_0,z_0,\sigma;C',C'',C'';A)dx_0dy_0dz_0d\sigma\\
\le& C\mathscr{W}(x,y,z,t;\frac 32C',\frac 32C'',\frac 32C'';A)\int_{A^{-\theta}}^1 (1+t)^{-2}\left(1+A^2t^2\right)^{-1/2}\sigma^{-1/2-\gamma/2} \left(1+A^2\sigma^2\right)^{\gamma/2}d\sigma\\
\le& C A^{\gamma} (1+t)^{-2}\left(1+A^2t^2\right)^{-1/2}
\cdot \mathscr{W}(x,y,z,t;\frac 32C',\frac 32C'',\frac 32C'';A).
\end{aligned}
\end{multline}
If $1<t \le 2$, there exist positive constants $\bar{C}_1$ and $\bar{C}_2$ such that
\begin{equation}\label{c12}
1/t,1/(1+t)\in(\bar{C}_1,\bar{C}_2),\ \ \ \ \ \ \ \ \bar{C}_1A^2<1+A^2t^2\le \bar{C}_2A^2.
\end{equation}
Thus, one has that
\begin{multline}\nonumber
\begin{aligned}
&\int_{A^{-\theta}}^1 \iiint_{\mathbb{R}^3}(t-\sigma)^{-2}\left(1+\frac{1}{12}A^2 (t-\sigma)^2\right)^{-1/2}\sigma^{-1/2-\gamma/2}
\left(1+A^2\sigma^2\right)^{-1/2+\gamma/2}\\[2mm]
&\cdot \mathbb{W}(x-x_0,y,z-z_0,t-\sigma;y_0;C_1,C_1,C_1;A)
\cdot \mathscr{W}(x_0,y_0,z_0,\sigma;C',C'',C'';A)dx_0dy_0dz_0d\sigma\\
\le& C\left(\int_{A^{-\theta}}^{t/2} t^{-2}\left(1+A^2t^2\right)^{-1/2}\sigma^{-1/2-\gamma/2} \left(1+A^2\sigma^2\right)^{\gamma/2}d\sigma\right.\\
&\left.+\int_{t/2}^1 (t-\sigma)^{-1/2}t^{-1/2-\gamma/2}
\left(1+A^2t^2\right)^{-1/2+\gamma/2}d\sigma\right)
\cdot \mathscr{W}(x,y,z,t;\frac 32C',\frac 32C'',\frac 32C'';A)\\
\le& C\left(A^{\gamma}(1+t)^{-2}
\left(1+A^2t^2\right)^{-1/2}+(1+t)^{-1/2-\gamma/2}\left(1+A^2t^2\right)^{-1/2+\gamma/2}\right)
\cdot \mathscr{W}(x,y,z,t;\frac 32C',\frac 32C'',\frac 32C'';A)\\
\le& C A^{\gamma}(1+t)^{-2}
\left(1+A^2t^2\right)^{-1/2}\cdot \mathscr{W}(x,y,z,t;\frac 32C',\frac 32C'',\frac 32C'';A)
\end{aligned}
\end{multline}
and it together with \eqref{2.37} verifies \eqref{nn2} for $t>1$.
\end{proof}

\begin{lemma}\label{lemma4} For positive constants $\gamma\in(0,1/2], C_1>1$ and $C',C''$ satisfying the requirements of Lemma \ref{lemmaz}-\ref{lemma0} and $t>1$, it holds that
\begin{multline}\label{nn3}
\begin{aligned}
&\int_{1}^t \iiint_{\mathbb{R}^3}(t-\sigma)^{-2}\left(1+\frac{1}{12}A^2 (t-\sigma)^2\right)^{-1/2}
(1+\sigma)^{-3}\left(1+A^2\sigma^2\right)^{-1+\gamma}
\\[2mm]
&\cdot \mathbb{W}(x-x_0,y,z-z_0,t-\sigma;y_0;C_1,C_1,C_1;A)
\cdot \mathscr{W}(x_0,y_0,z_0,\sigma;C',C'',C'';A)dx_0dy_0dz_0d\sigma\\
=&O(1)(1+t)^{-2} \left(1+A^2t^2\right)^{-1/2}\mathscr{W}(x,y,z,t;\frac 32C',\frac 32C'',\frac 32C'';A).
\end{aligned}
\end{multline}
\end{lemma}

\begin{proof}
Lemma \ref{lemmaz} and Lemma \ref{lemma0} yield that
\begin{multline}\label{240}
\begin{aligned}
&\int_{1}^t \iiint_{\mathbb{R}^3}(t-\sigma)^{-2}\left(1+\frac{1}{12}A^2 (t-\sigma)^2\right)^{-1/2}
(1+\sigma)^{-3}\left(1+A^2\sigma^2\right)^{-1+\gamma}
\\[2mm]
&\cdot \mathbb{W}(x-x_0,y,z-z_0,t-\sigma;y_0;C_1,C_1,C_1;A)
\cdot \mathscr{W}(x_0,y_0,z_0,\sigma;C',C'',C'';A)dx_0dy_0dz_0d\sigma\\[2mm]
=&\left(\int_1^{\max\{1,t/2\}}+\int_{\max\{1,t/2\}}^t\right) \iiint_{\mathbb{R}^3}(t-\sigma)^{-2}\left(1+\frac{1}{12}A^2 (t-\sigma)^2\right)^{-1/2}
(1+\sigma)^{-3}\left(1+A^2\sigma^2\right)^{-1+\gamma}
\\[2mm]
&\cdot \mathbb{W}(x-x_0,y,z-z_0,t-\sigma;y_0;C_1,C_1,C_1;A)
\cdot \mathscr{W}(x_0,y_0,z_0,\sigma;C',C'',C'';A)dx_0dy_0dz_0d\sigma\\[2mm]
=&O(1)\left(\int_1^{\max\{1,t/2\}}(1+\sigma)^{-3/2}\left(1+A^2\sigma^2\right)^{-1/2+\gamma}d\sigma \right)
\cdot (1+t)^{-2}\left(1+A^2t^2\right)^{-1/2}
\mathscr{W}(x,y,z,t;\frac 32C',\frac 32C'',\frac 32C'';A)\\
&+O(1)\left(\int_{\max\{1,t/2\}}^t(t-\sigma)^{-1/2}d\sigma\right)\cdot (1+t)^{-3}\left(1+A^2t^2\right)^{-1+\gamma}
\mathscr{W}(x,y,z,t;\frac 32C',\frac 32C'',\frac 32C'';A)\\
=&O(1)A^{-1+2\gamma}(1+t)^{-2} \left(1+A^2t^2\right)^{-1/2}
\mathscr{W}(x,y,z,t;\frac 32C',\frac 32C'',\frac 32C'';A).
\end{aligned}
\end{multline}
Here, one uses the assumptions $\gamma\in(0,1/2]$ and $t\ge\sigma\ge1$ to ensure that
$$\left(1+A^2t^2\right)^{-1/2+\gamma}\le \left(1+A^2\sigma^2\right)^{-1/2+\gamma}\le A^{-1+2\gamma}\le 1,\ \ \ \ \ \ \ \ t^{-3/2}\le 2\sqrt{2}\left(1+t\right)^{-3/2}.$$
\end{proof}

\subsection{Computational Lemmas: Tools for Pointwise Estimates of $c(x,y,z,t)$}
The following lemmas are prepared to estimate $c(x,y,z,t)$ based on its representations \eqref{c1} and \eqref{c0}, space-time structure of Green's functions $\mathbb{G}_{c,1}(x,y,z,t;y_0), \mathbb{G}_{c,0}(x,y,z)$ and ansatz assumption of $n(x,y,z,t)$:
\begin{lemma}\label{cp1}
\textbf{(Parabolic-Parabolic case.)} For positive constants $\alpha\in[3/2,2], \gamma\in(0,1), C_1>1$ and $C'_1,C''_1$ satisfying the requirements for $C',C''$ in Lemma \ref{lemmaz}-\ref{lemma0}, it holds that
\begin{multline}\label{time3}
\begin{aligned}
&\int_{\min\left\{1, t\right\}}^t \iiint_{\mathbb{R}^3}e^{-t+\sigma}(t-\sigma)^{-\alpha}\left(1+A^2 (t-\sigma)^2\right)^{-1/2}
(1+\sigma)^{-3/2}\left(1+A^2\sigma^2\right)^{-1/2+\gamma/2}
\\[2mm]
&\cdot \mathbb{W}(x-x_0,y,z-z_0,t-\sigma;y_0;C_1,C_1,C_1;A)
\cdot \mathscr{W}(x_0,y_0,z_0,\sigma;C'_1,C''_1,C''_1;A)dx_0dy_0dz_0d\sigma\\
=&O(1)\left(e^{-t/2}(1+t)^{-\alpha+1}+(1+t)^{-3/2}\right) \left(1+A^2t^2\right)^{-1/2+\gamma/2}
\mathscr{W}(x,y,z,t;\frac 32C'_1,\frac 32C''_1,\frac 32C''_1;A),
\end{aligned}
\end{multline}
and
\begin{multline}\label{time2}
\begin{aligned}
&\int_{\min\left\{A^{-\theta},t\right\}}^{\min\left\{1, t\right\}} \iiint_{\mathbb{R}^3}e^{-t+\sigma}(t-\sigma)^{-\alpha}\left(1+A^2 (t-\sigma)^2\right)^{-1/2}\sigma^{-1/4-\gamma/4}\left(1+A^2\sigma^2\right)^{-1/4+\gamma/4}
\\[2mm]
&\cdot \mathbb{W}(x-x_0,y,z-z_0,t-\sigma;y_0;C_1,C_1,C_1;A)
\cdot \mathscr{W}(x_0,y_0,z_0,\sigma;C'_1,C''_1,C''_1;A)dx_0dy_0dz_0d\sigma\\
=&O(1)\mathscr{A}_3(t;\alpha)\mathscr{W}(x,y,z,t;\frac 32C'_1,\frac 32C''_1,\frac 32C''_1;A)
\end{aligned}
\end{multline}
with $\mathscr{A}_3(t;\alpha)$ defined by
\begin{equation}\label{aa3}
\mathscr{A}_3(t;\alpha)\equiv \begin{cases}
0,\ \ \ \ &\text{ for }\ 0< t\le A^{-\theta},\\
t^{-1/4-\gamma/4}\left(1+A^2t^2\right)^{-1/4+\gamma/4},\ \ \ \ &\text{ for }\ A^{-\theta}< t\le 1,\\
A^{1/2+\gamma/2}e^{-t/2}(1+t)^{-\alpha}\left(1+A^2t^2\right)^{-1/2},\ \ \ \ &\text{ for }\ t>1.
\end{cases}
\end{equation}
\end{lemma}

\begin{proof}
The proof of this lemma is similar to the ones for Lemma \ref{lemma3}-\ref{lemma4}: If $t\le A^{-\theta}$, the integral in \eqref{time2} equals to 0. For $t\in(A^{-\theta},1]$, combining Lemma \ref{lemmaz}-\ref{lemma0} one has that
\begin{multline}\nonumber
\begin{aligned}
&\int_{A^{-\theta}}^t \iiint_{\mathbb{R}^3}e^{-t+\sigma}(t-\sigma)^{-\alpha}\left(1+\frac{1}{12}A^2 (t-\sigma)^2\right)^{-1/2}\sigma^{-1/4-\gamma/4}
\left(1+A^2\sigma^2\right)^{-1/4+\gamma/4}
\\[2mm]
&\cdot \mathbb{W}(x-x_0,y,z-z_0,t-\sigma;y_0;C_1,C_1,C_1;A)
\cdot \mathscr{W}(x_0,y_0,z_0,\sigma;C'_1,C''_1,C''_1;A)dx_0dy_0dz_0d\sigma\\
\le& C\mathscr{W}(x,y,z,t;\frac 32C'_1,\frac 32C''_1,\frac 32C''_1;A)
\cdot\left(\int_{A^{-\theta}}^{\max\{A^{-\theta},t/2\}} e^{-t/2}t^{-\alpha+1}\left(1+A^2 t^2\right)^{-1/2}\sigma^{-1/4-\gamma/4}
\left(1+A^2\sigma^2\right)^{1/4+\gamma/4}d\sigma\right.\\
&\left.+\int_{\max\{A^{-\theta},t/2\}}^t e^{-t+\sigma}(t-\sigma)^{-\alpha+3/2}t^{-1/4-\gamma/4}
\left(1+A^2t^2\right)^{-1/4+\gamma/4}d\sigma\right)\\
\le &C t^{-1/4-\gamma/4}\left(1+A^2t^2\right)^{-1/4+\gamma/4}
\mathscr{W}(x,y,z,t;\frac 32C'_1,\frac 32C''_1,\frac 32C''_1;A)
\end{aligned}
\end{multline}
and verifies \eqref{time2} for $t\in(A^{-\theta},1]$.

\vskip .05in

If $1<t \le 2$, from \eqref{c12} one has that
\begin{multline}\label{244}
\begin{aligned}
&\int_{A^{-\theta}}^1 \iiint_{\mathbb{R}^3} e^{-t+\sigma}(t-\sigma)^{-\alpha}\left(1+\frac{1}{12}A^2 (t-\sigma)^2\right)^{-1/2}\sigma^{-1/4-\gamma/4}
\left(1+A^2\sigma^2\right)^{-1/4+\gamma/4}\\[2mm]
&\cdot \mathbb{W}(x-x_0,y,z-z_0,t-\sigma;y_0;C_1,C_1,C_1;A)
\cdot \mathscr{W}(x_0,y_0,z_0,\sigma;C'_1,C''_1,C''_1;A)dx_0dy_0dz_0d\sigma\\
\le &C\left(\int_{A^{-\theta}}^{t/2}e^{-t/2} t^{-\alpha}\left(1+A^2t^2\right)^{-1/2}\sigma^{-1/4-\gamma/4} \left(1+A^2\sigma^2\right)^{1/4+\gamma/4}d\sigma\right.\\
&\left.+\int_{t/2}^1 e^{-t+\sigma}(t-\sigma)^{-\alpha+3/2}t^{-1/4-\gamma/4}
\left(1+A^2t^2\right)^{-1/4+\gamma/4}d\sigma\right)
\cdot \mathscr{W}(x,y,t;\frac 32C'_1,\frac 32C''_1)\\
\le& C\left(A^{1/2+\gamma/2}e^{-t/2}(1+t)^{-\alpha}
\left(1+A^2t^2\right)^{-1/2}+(1+t)^{-1/4-\gamma/4}\left(1+A^2t^2\right)^{-1/4+\gamma/4}\right)
\cdot \mathscr{W}(x,y,z,t;\frac 32C'_1,\frac 32C''_1,\frac 32C''_1;A)\\
\le& C A^{1/2+\gamma/2}e^{-t/2}(1+t)^{-\alpha}
\left(1+A^2t^2\right)^{-1/2}\cdot \mathscr{W}(x,y,z,t;\frac 32C'_1,\frac 32C''_1,\frac 32C''_1;A).
\end{aligned}
\end{multline}
If $t>2$ and $\sigma\le 1$, one has that $t-\sigma>t/2>(1+t)/4$ and thus
\begin{multline}\nonumber
\begin{aligned}
&\int_{A^{-\theta}}^1 \iiint_{\mathbb{R}^2}e^{-t+\sigma}(t-\sigma)^{-\alpha}\left(1+\frac{1}{12}A^2 (t-\sigma)^2\right)^{-1/2}\sigma^{-1/4-\gamma/4}
\left(1+A^2\sigma^2\right)^{-1/4+\gamma/4}\\[2mm]
&\cdot \mathbb{W}(x-x_0,y,z-z_0,t-\sigma;y_0;C_1,C_1,C_1;A)
\cdot \mathscr{W}(x_0,y_0,z_0,\sigma;C'_1,C''_1,C''_1;A)dx_0dy_0dz_0d\sigma\\
\le& C \mathscr{W}(x,y,z,t;\frac 32C'_1,\frac 32C''_1,\frac 32C''_1;A)\int_{A^{-\theta}}^1 e^{-t/2}(1+t)^{-\alpha}\left(1+A^2t^2\right)^{-1/2}\sigma^{-1/4-\gamma/4} \left(1+A^2\sigma^2\right)^{1/4+\gamma/4}d\sigma\\
\le& C A^{1/2+\gamma/2} e^{-t/2}(1+t)^{-\alpha}\left(1+A^2t^2\right)^{-1/2}
\cdot \mathscr{W}(x,y,z,t;\frac 32C'_1,\frac 32C''_1,\frac 32C''_1;A),
\end{aligned}
\end{multline}
since $\left(1+A^2\sigma^2\right)^{1/4+\gamma/4}\le C A^{1/2+\gamma/2}$ for $\sigma\le 1$. It together with \eqref{244} verifies \eqref{time2} for $t>1$.

For \eqref{time3}, one only needs to consider $t>1$ and similar to \eqref{240}, one has that
\begin{multline}\nonumber
\begin{aligned}
&\int_{1}^t \iiint_{\mathbb{R}^3}e^{-t+\sigma}(t-\sigma)^{-\alpha}\left(1+A^2 (t-\sigma)^2\right)^{-1/2}
(1+\sigma)^{-3/2}\left(1+A^2\sigma^2\right)^{-1/2+\gamma/2}
\\[2mm]
&\cdot \mathbb{W}(x-x_0,y,z-z_0,t-\sigma;y_0;C_1,C_1,C_1;A)
\cdot \mathscr{W}(x_0,y_0,z_0,\sigma;C'_1,C''_1,C''_1;A)dx_0dy_0dz_0d\sigma\\
=&O(1)\int_1^{\max\{1,t/2\}}\left(1+A^2\sigma^2\right)^{\gamma/2}d\sigma
\cdot e^{-t/2}(1+t)^{-\alpha}\left(1+A^2t^2\right)^{-1/2}
\mathscr{W}(x,y,z,t;\frac 32C'_1,\frac 32C''_1,\frac 32C''_1;A)\\
&+O(1)\int_{\max\{1,t/2\}}^t e^{-t+\sigma}(t-\sigma)^{-\alpha+3/2}d\sigma\cdot (1+t)^{-1}\left(1+A^2t^2\right)^{-1/2+\gamma/2}
\mathscr{W}(x,y,z,t;\frac 32C'_1,\frac 32C''_1,\frac 32C''_1;A)\\
\le& C\left(1+A^2 t^2\right)^{\gamma/2}
\cdot e^{-t/2}(1+t)^{-\alpha+1}\left(1+A^2t^2\right)^{-1/2}
\mathscr{W}(x,y,z,t;\frac 32C'_1,\frac 32C''_1,\frac 32C''_1;A)\\
&+C (1+t)^{-3/2}\left(1+A^2t^2\right)^{-1/2+\gamma/2}
\mathscr{W}(x,y,z,t;\frac 32C'_1,\frac 32C''_1,\frac 32C''_1;A)\\
=&O(1)\left(e^{-t/2}(1+t)^{-\alpha+1}+(1+t)^{-3/2}\right) \left(1+A^2t^2\right)^{-1/2+\gamma/2}
\mathscr{W}(x,y,z,t;\frac 32C'_1,\frac 32C''_1,\frac 32C''_1;A)
\end{aligned}
\end{multline}
which verifies \eqref{time3}.
\end{proof}

\begin{lemma}\label{cp2}
\textbf{(Parabolic-Elliptic case.)}
For constants $C'_1, C''_1\ge 60$, it holds that
\begin{multline}\label{2.39}
\begin{aligned}
&\iiint_{\mathbb{R}^3}\left(1+\frac{1}{\sqrt{(x-x_0)^2+(y-y_0)^2+(z-z_0)^2}} \right)\frac{e^{-\sqrt{(x-x_0)^2+(y-y_0)^2+(z-z_0)^2}}}{\sqrt{(x-x_0)^2+(y-y_0)^2+(z-z_0)^2}}\\[2mm] &\cdot\mathscr{W}(x_0,y_0,z_0,t;C'_1,C''_1,C''_1;A)dx_0dy_0dz_0\\
\le& C\mathscr{W}(x,y,z,t;\frac 32C'_1,\frac 32C''_1,\frac 32C''_1;A).
\end{aligned}
\end{multline}
\end{lemma}

\begin{proof}In the region $\left|x-\frac{At}2 y-\left(x_0-\frac{At}2y_0\right)\right|\ge At\left|y-y_0\right|$, one has that
$$
\left|x-x_0\right|=\left|x-\frac{At}2 y-\left(x_0-\frac{At}2y_0\right)+\frac{At}2\left(y-y_0\right)\right|\ge \left|x-\frac{At}2 y-\left(x_0-\frac{At}2y_0\right)\right|/2,
$$
and thus
$$
\exp\left(-\frac{|x-x_0|}{3}\right)\le \exp\left(\frac{-\left|x-\frac{At}2 y-\left(x_0-\frac{At}2y_0\right)\right|}{6}\right).$$
And if $\left|x-\frac{At}2 y-\left(x_0-\frac{At}2y_0\right)\right|< At\left|y-y_0\right|$, it holds that
$$
\exp\left(-\frac{|y-y_0|}{6}\right)\le \exp\left(-\frac{\left|x-\frac{At}2 y-\left(x_0-\frac{At}2y_0\right)\right|}{6At}\right).$$
Then one always has that
\begin{equation}\label{expo1}
\exp\left(-\frac{|x-x_0|}{3}-\frac{|y-y_0|}{3}\right)\le \exp\left(-\frac{|y-y_0|}{6}-\frac{\left|x-\frac{At}2 y-\left(x_0-\frac{At}2y_0\right)\right|}{6(1+At)}\right),
\end{equation}
and thus
\begin{multline}\nonumber
\begin{aligned}
&\iiint_{\mathbb{R}^3}\left(1+\frac{1}{\sqrt{(x-x_0)^2+(y-y_0)^2+(z-z_0)^2}} \right)\frac{e^{-\sqrt{(x-x_0)^2+(y-y_0)^2+(z-z_0)^2}}}{\sqrt{(x-x_0)^2+(y-y_0)^2+(z-z_0)^2}}\\[2mm] &\cdot\mathscr{W}(x_0,y_0,z_0,t;C'_1,C''_1,C''_1;A)dx_0dy_0dz_0\\
\le& \iiint_{\mathbb{R}^3}\left(1+\frac{1}{\sqrt{(x-x_0)^2+(y-y_0)^2+(z-z_0)^2}} \right)\frac{e^{-\sqrt{(x-x_0)^2+(y-y_0)^2+(z-z_0)^2}/3}}{\sqrt{(x-x_0)^2+(y-y_0)^2+(z-z_0)^2}}\\
&\cdot \exp\left(-\frac{|x-x_0|}{3}-\frac{|y-y_0|}{3}-\frac{|z-z_0|}{3}\right)\cdot\mathscr{W}(x_0,y_0,z_0,t;C'_1,C''_1,C''_1;A)dx_0dy_0dz_0\\
\le&
\iiint_{\mathbb{R}^3}\left(1+\frac{1}{\sqrt{(x-x_0)^2+(y-y_0)^2+(z-z_0)^2}} \right)\frac{e^{-\sqrt{(x-x_0)^2+(y-y_0)^2+(z-z_0)^2}/3}}{\sqrt{(x-x_0)^2+(y-y_0)^2+(z-z_0)^2}}\\
&\cdot
\left(\exp\left(-\frac{|y-y_0|}6-\frac{\left|x-\frac{At}2 y-\left(x_0-\frac{At}2y_0\right)\right|}{6(1+At)}\right)
\mathscr{W}^x(x_0,y_0,t;C'_1,C'_1;A)\mathscr{W}^o(y_0,t;C''_1,C''_1)\right)dx_0dy_0\\
 &\cdot \left(e^{-|z-z_0|/3}\cdot\mathscr{W}^o(z_0,t;C''_1,C''_1)\right)dz_0.
\end{aligned}
\end{multline}
In the above integral, we choose to integrate
$$
\left(1+\frac{1}{\sqrt{(x-x_0)^2+(y-y_0)^2+(z-z_0)^2}} \right)\frac{e^{-\sqrt{(x-x_0)^2+(y-y_0)^2+(z-z_0)^2}}}{\sqrt{(x-x_0)^2+(y-y_0)^2+(z-z_0)^2}}
$$
with respect to $x_0, y_0, z_0$. It is integrable and also no extra factors about time come out after integration; the other terms will be estimated after a decomposition of $\mathbb{R}^2$: if $x_0, y_0$ satisfy $\left|x-\frac{At}2 y\right|\ge \frac{11}{10}\left|x_0-\frac{At}2 y_0\right|$ and $\left|y\right|\ge \frac{11}{10}\left|y_0\right|$, it holds that
$$
\exp\left(-\frac{|y-y_0|}6-\frac{\left|x-\frac{At}2 y-\left(x_0-\frac{At}2y_0\right)\right|}{6(1+At)}\right)\le \exp\left(-\frac{|y|}{66}-\frac{\left|x-\frac{At}2 y\right|}{66(1+At)}\right),
$$
and if $\left|x-\frac{At}2 y\right|\ge \frac{11}{10}\left|x_0-\frac{At}2 y_0\right|$ and $\left|y\right|\le \frac{11}{10}\left|y_0\right|$, it holds that
$$
e^{-\frac{\left|x-\frac{At}2 y-\left(x_0-\frac{At}2y_0\right)\right|}{6(1+At)}}\mathscr{W}^o(y_0,t;C''_1,C''_1)\le e^{-\frac{\left|x-\frac{At}2 y\right|}{66(1+At)}}\mathscr{W}^o(y,t;\frac{121}{100}C''_1,\frac{11}{10}C''_1).
$$
Similarly, for $\left\{(x_0,y_0): \left|x-\frac{At}2 y\right|\le \frac{11}{10}\left|x_0-\frac{At}2 y_0\right|, \left|y\right|\ge \frac{11}{10}\left|y_0\right|\right\}$ and $\left\{(x_0,y_0): \left|x-\frac{At}2 y\right|\le \frac{11}{10}\left|x_0-\frac{At}2 y_0\right|, \left|y\right|\le \frac{11}{10}\left|y_0\right|\right\}$, one has
$$
e^{-\frac{|y-y_0|}6}\mathscr{W}^x(x_0,y_0,t;C'_1,C'_1;A)\le e^{-\frac{|y|}{66}}\mathscr{W}^x(x,y,t;\frac{121}{100}C'_1,\frac{11}{10}C'_1;A),
$$
and
$$
\mathscr{W}^x(x_0,y_0,t;C'_1,C'_1;A)\cdot \mathscr{W}^o(y_0,t;C''_1,C''_1)
\le \mathscr{W}^x(x,y,t;\frac{121}{100}C'_1,\frac{11}{10}C'_1;A)\cdot \mathscr{W}^o(y,t;\frac{121}{100}C'_1,\frac{11}{10}C'_1)
$$
respectively. Based on a similar but much simpler decomposition of $\mathbb{R}$ for $z_0$, one also has that
$$e^{-|z-z_0|/3}\cdot \mathscr{W}^o(z_0,t;C''_1,C''_1)\le \mathscr{W}^o(z,t;\frac{121}{100}C'_1,\frac{11}{10}C'_1).$$
Thus,
\begin{multline}
\begin{aligned}
&\iiint_{\mathbb{R}^3}\left(1+\frac{1}{\sqrt{(x-x_0)^2+(y-y_0)^2+(z-z_0)^2}} \right)\frac{e^{-\sqrt{(x-x_0)^2+(y-y_0)^2+(z-z_0)^2}}}{\sqrt{(x-x_0)^2+(y-y_0)^2+(z-z_0)^2}}\\[2mm] &\cdot\mathscr{W}(x_0,y_0,z_0,t;C'_1,C''_1,C''_1;A)dx_0dy_0dz_0\\
\le &C\mathscr{W}(x,y,z,t;\frac 32C'_1,\frac 32C''_1,\frac 32C''_1;A)
\\
&\cdot\iiint_{\mathbb{R}^3}\left(1+\frac{1}{\sqrt{(x-x_0)^2+(y-y_0)^2+(z-z_0)^2}} \right)\frac{e^{-\sqrt{(x-x_0)^2+(y-y_0)^2+(z-z_0)^2}/3}}{\sqrt{(x-x_0)^2+(y-y_0)^2+(z-z_0)^2}}\\
&\cdot\mathscr{W}(x_0,y_0,z_0,t;C'_1,C''_1,C''_1;A)dx_0dy_0dz_0\\
\le& C\mathscr{W}(x,y,z,t;\frac 32C'_1,\frac 32C''_1,\frac 32C''_1;A)\
\end{aligned}
\end{multline}
which verifies \eqref{2.39}.
\end{proof}

\setcounter{equation}{0}

\section{Global existence, nonlinear stability and pointwise structure}

In this section, we first show the local existence and blow-up criterion. The criterion shows that the solutions for \eqref{pro} exist globally if the $L^\infty$-norm is controlled. The $L^\infty$ boundness could be proved by Green's function method. Meanwhile, one also obtains the pointwise space-time structures of the solutions and finished the proofs of Theorem \ref{main} and \ref{main0}.

\subsection{Local well-posedness and regularity criterion}

The local well-posedness and $L^\infty$-criterion of the problem \eqref{pro} could be established as follows:
\begin{pro}\label{local}
Let $\epsilon=1$ and initial functions $n_0, c_0\ge 0, n_0, \nabla c_0\in H^s\cap W^{s,\infty}(\mathbb{R}^3) (s\ge1)$. Then there exists $T=T(n_0,c_0)>0$ such that the non-negative solution of \eqref{pro} satisfies
$$n(x,y,t), \nabla c(x,y,t)\in C([0,T],H^s\cap W^{s,\infty}(\mathbb{R}^3)).$$
Moreover, for any given $T_0>0$, if the solution verifies the following condition
$$\lim_{t->T_0} \sup_{0\le \tau\le t}\|n(\cdot,\cdot,\cdot,\tau)\|_{L^\infty}<\infty,$$
it could be extended up to time $T+\delta$ for a sufficiently small positive constant $\delta$. Furthermore, if $n_0\in L^1(\mathbb{R}^3)$, the $L^1$ norm of the solution is preserved for all time.
\end{pro}

\begin{remark}
Proposition \ref{local} still holds for the case $\epsilon=0$.
\end{remark}

Proposition \ref{local} could be proved by following the standard method and we omit the details here. One could refer to \cite{agm,clm} for a standard proof of the local existence and \cite{h} for regularity criterion.
Proposition \ref{local} shows that to prove the global existence, one only needs to control the $L^\infty$-norm of the solutions and it will be done by Green's function method in the rest of the paper.

\subsection{Representations of solutions and linear estimates}
The solution $n(x,y,z,t)$ of \eqref{pro} could be represented as follows:
\begin{multline}\label{pre}
\begin{aligned}
n(x,y,z,t)=&\iiint_{\mathbb{R}^3}\mathbb{G}(x-x_0,y,z-_0,t;y_0) n_0(x_0,y_0,z_0)dx_0dy_0dz_0\\
&+\int_0^t \iiint_{\mathbb{R}^3}\mathbb{G}(x-x_0,y,z-z_0,t-\sigma;y_0) \nabla \left(n\nabla c\right)(x_0,y_0,z_0)d x_0 dy_0dz_0d\sigma,
\end{aligned}
\end{multline}
in which $c(x,y,z,t)$ could be solved directly from the second equation of \eqref{pro} by Green's functions $\mathbb{G}_{c,1}(x,y,z,t;y_0)$ or $\mathbb{G}_{c,0}(x,y,z)$ with the representations \eqref{c1} and \eqref{c0} respectively.

\vskip .05in

First, consider the structure of initial wave propagation. From Lemma \ref{lemma1} and \eqref{ini}, there exists a positive constant $N_0$ independent of $C_0$ and $A$ such that
\begin{multline}\label{linn}
\begin{aligned}
&\left|\iiint_{\mathbb{R}^3}\mathbb{G}(x-x_0,y,z-z_0,t;y_0) n_0(x_0,y_0,z_0)dx_0dy_0dz_0\right|\\
=&O(1)\iiint_{\mathbb{R}^3}\mathbb{G}(x-x_0,y,z-z_0,t;y_0) \exp{\left(-\frac{|x_0|+|y_0|+|z_0|}{C_*}\right)} dx_0dy_0dz_0\\
\le& \frac 12N_0C_0 t^{-1/2+\beta}(1+t)^{-1-\beta}\left(1+A^2t^2\right)^{-1/2+\beta}
\mathscr{W}^x(x,y,t;16,6C_*;A)\mathscr{W}^o(y,t;9,9C_*)\mathscr{W}^o(z,t;9,4C_*)
\end{aligned}
\end{multline}
for any constant $\beta$ satisfying $0\le\beta\le \frac 12$ and wave structure functions $\mathscr{W}^x(x,y,t;D_1,D_2;A), \mathscr{W}^o(y,t;D_1,D_2)$ defined by \eqref{wavexx} and \eqref{waveyz} respectively. Based on this estimate, one could pose ansatz assumptions as follows:
\begin{equation}
\label{ansatz}
\left|n(x,y,t)\right|\le 2N_0 C_0 \mathscr{A}(t;A,\theta,\gamma) \mathscr{W}(x,y,z,t;C'_1,C''_1,C''_1;A)
\end{equation}
with positive constants $C'_1, C''_1\ge \max\{16C_1, 9C_*, 60\}$ and $\frac 35C'_1, \frac 35C''_1$ satisfying the requirements for $C', C''$ in Lemma \ref{lemmaz}-\ref{lemma0} (i.e. $\frac 35 C''_1>C_1$ with $C_1$ in \eqref{greenpx}-\eqref{greenpz} while \eqref{constant1} and \eqref{constant2} satisfied by $\frac 35C'_1, \frac 35C''_1$ taking place of $C', C''$ and this could be done by taking $C'_1, C''_1$ large).
Here, the wave structure function $\mathscr{W}(x,y,z,t;D_1,D_2,D_3;A)$ with positive constants $D_1, D_2, D_3$ is given by \eqref{wavest} and the decaying rate function $\mathscr{A}(t;A,\theta,\gamma)$ is defined as follows
$$
\mathscr{A}(t;A,\theta,\gamma)\equiv \begin{cases}
1,\ \ \ \ &\text{ for }\ 0< t\le A^{-\theta},\\
A^{-(1-\theta)\gamma}t^{-1/4-\gamma/4}\left(1+A^2t^2\right)^{-1/4+\gamma/4},\ \ \ \ &\text{ for }\ A^{-\theta}< t\le 1,\\
A^{-(1-\theta)\gamma}(1+t)^{-3/2}\left(1+A^2t^2\right)^{-1/2+\gamma/2},\ \ \ \ &\text{ for }\ t>1
\end{cases}
$$
with positive constants $\gamma, \theta$ satisfying
\begin{equation}\label{par11pp}
\frac 23<\theta<1,\ \ \ \ \ \ \ \ \ 1/3< \gamma\le 1/2,
\end{equation}
for Parabolic-Parabolic case and
\begin{equation}\label{par11pe}
\frac 23<\theta<1,\ \ \ \ \ \ \ \ \ 0< \gamma\le 1/2,
\end{equation}
for Parabolic-Elliptic case.

\subsection{Proof of Theorem \ref{main}: Parabolic-Parabolic Case}

To justify the ansatz assumption \eqref{ansatz} through the solution presentation \eqref{pre} for $n(x,y,z,t)$, one still needs a pointwise structure of $\nabla c(x,y,z,t)$:
\begin{lemma}\label{cc1}
\textbf{(Pointwise structures of $c(x,y,z,t)$ and its derivatives.)}
There exists a positive constant $N_{c1}$ independent of $C_0$ and $A$ such that
\begin{multline}\nonumber
\begin{aligned}
\left| c(x,y,z,t)\right|, \left|\nabla c(x,y,z,t)\right|\le &N_{c1}C_0 \mathscr{W}(x,y,z,t;\frac 32C'_1, \frac 32C''_1,\frac 32C''_1;A)\\[2mm]
&\cdot{ {\begin{cases}
1,\ \ \ \ &\text{ for }\ 0< t\le A^{-\theta},\\
A^{-(1-\theta)\gamma}t^{-1/4-\gamma/4}\left(1+A^2t^2\right)^{-1/4+\gamma/4},\ \ \ \ &\text{ for }\ A^{-\theta}< t\le 1,\\
A^{1/2-\gamma/2}\cdot A^{-(1-\theta)\gamma}(1+t)^{-3/2}\left(1+A^2t^2\right)^{-1/2+\gamma/2},\ \ \ \ &\text{ for }\ t>1.
\end{cases}}}
\end{aligned}
\end{multline}
\end{lemma}

\begin{proof}
For Parabolic-Parabolic case, i.e. $\epsilon\equiv 1$ in \eqref{pro}, the solution $c(x,y,z,t)$ could be represented by \eqref{c1} and thus
\begin{multline}\label{359}
\begin{aligned}
\nabla c(x,y,z,t)
=&\iiint_{\mathbb{R}^3}\nabla\mathbb{G}_{c,1}(x-x_0,y,z-z_0,t;y_0)c(x_0,y_0,z_0,0)dx_0dy_0dz_0\\
&+\int_0^t\iiint_{\mathbb{R}^3}
\nabla\mathbb{G}_{c,1}(x-x_0,y,z-z_0,t-\sigma;y_0)n(x_0,y_0,z_0,\sigma)dx_0dy_0dz_0d\sigma.
\end{aligned}
\end{multline}
A direct computation yields that
$$\begin{cases}
\partial_x\mathbb{G}_{c,1}(x-x_0,y,z-z_0,t;y_0)=-\partial_{x_0}\mathbb{G}_{c,1}(x-x_0,y,z-z_0,t;y_0),\\
\partial_y\mathbb{G}_{c,1}(x-x_0,y,z-z_0,t;y_0)=-\partial_{y_0}\mathbb{G}_{c,1}(x-x_0,y,z-z_0,t;y_0)+At\partial_{x_0}\mathbb{G}_{c,1}(x-x_0,y,z-z_0,t;y_0),\\
\partial_z\mathbb{G}_{c,1}(x-x_0,y,z-z_0,t;y_0)=-\partial_{z_0}\mathbb{G}_{c,1}(x-x_0,y,z-z_0,t;y_0),
\end{cases}$$
and thus similar to Lemma \ref{lemma1}, one has that
\begin{multline}\label{cin2}
\begin{aligned}
&\left|\iiint_{\mathbb{R}^3}\nabla\mathbb{G}_{c,1}(x-x_0,y,z-z_),t;y_0)c(x_0,y_0,z_0,0)dx_0dy_0dz_0\right|\\
\le& C\left(1+At\right)\left|\iiint_{\mathbb{R}^3}\mathbb{G}_{c,1}(x-x_0,y,z-z_),t;y_0)\nabla c(x_0,y_0,z_0,0)dx_0dy_0dz_0\right|\\
\le& CC_0^* \mathscr{W}^x(x,y,t;16C_1,6C_*;A)\mathscr{W}^o(y,t;9C_1,9C_*)\mathscr{W}^o(z,t;9C_1,4C_*)\\
&\cdot e^{-t}(1+t)^{-1}
\begin{cases}
2,\ \ \ \ &\text{ for }\ 0< t\le A^{-1},\\
2 At\cdot t^{-1/2}\left(1+A^2t^2\right)^{-1/2} ,\ \ \ \ &\text{ for }\ t>A^{-1}
\end{cases}
\end{aligned}
\end{multline}
for a positive constant $C$ independent of $C_0^*$ and $A$.
From \eqref{cin2} and \eqref{inic}, one could take $A$ to satisfy $C_0^*A<C_0/2$ and has that
\begin{equation}\label{cin}
\left|\iiint_{\mathbb{R}^3}\nabla\mathbb{G}_{c,1}(x-x_0,y,z-z_0,t;y_0)c(x_0,y_0,z_0,0)dx_0dy_0dz_0\right|
\le N_1C_0 e^{-t/2}\mathscr{A}(t;A,\theta,\gamma) \mathscr{W}(x,y,z,t;C'_1,C''_1,C''_1;A),
\end{equation}
for a constant $N_1>0$ independent of $C_0$ and $A$.

\vskip .05in

The remainder could be estimated by Lemma \ref{lemma2} and Lemma \ref{cp1} with $\alpha=2$:
\begin{equation}\label{361}
\begin{aligned}
&\left|\int_0^t\iiint_{\mathbb{R}^3}
\nabla\mathbb{G}_{c,1}(x-x_0,y,z-z_0,t-\sigma;y_0)n(x_0,y_0,z_0,\sigma)dx_0dy_0dz_0d\sigma\right|\\
=&O(1)\left(N_0C_0\right)\int_0^{\min\left\{A^{-\theta},t\right\}} \iiint_{\mathbb{R}^3}e^{-t+\sigma}(t-\sigma)^{-2}\left(1+A^2 (t-\sigma)^2\right)^{-1/2}\\
&\cdot \mathbb{W}(x-x_0,y,z-z_0,t-\sigma;y_0;C_1,C_1,C_1;A)\cdot
\mathscr{W}(x_0,y_0,z_0,\sigma;C'_1,C''_1,C''_1;A)dx_0dy_0dz_0d\sigma\\
&+O(1)A^{-(1-\theta)\gamma}\left(N_0C_0\right)\int_{\min\left\{A^{-\theta},t\right\}}^{\min\left\{1, t\right\}} \iiint_{\mathbb{R}^3}e^{-t+\sigma}(t-\sigma)^{-2}\left(1+A^2 (t-\sigma)^2\right)^{-1/2}\sigma^{-1/4-\gamma/4}\left(1+A^2\sigma^2\right)^{-1/4+\gamma/4}\\
&\cdot \mathbb{W}(x-x_0,y,z-z_0,t-\sigma;y_0;C_1,C_1,C_1;A)\cdot \mathscr{W}(x_0,y_0,z_0,\sigma;C'_1,C''_1,C''_1;A)dx_0dy_0dz_0d\sigma\\
&+O(1)A^{-(1-\theta)\gamma}\left(N_0C_0\right)\int_{\min\left\{1, t\right\}}^t \iiint_{\mathbb{R}^3}e^{-t+\sigma}(t-\sigma)^{-2}\left(1+A^2 (t-\sigma)^2\right)^{-1/2}
(1+\sigma)^{-3/2}\left(1+A^2\sigma^2\right)^{-1/2+\gamma/2}
\\
&\cdot \mathbb{W}(x-x_0,y,z-z_0,t-\sigma;y_0;C_1,C_1,C_1;A)
\cdot \mathscr{W}(x_0,y_0,z_0,\sigma;C'_1,C''_1,C''_1;A)dx_0dy_0dz_0d\sigma\\
=&O(1)\left(N_0C_0\right) \mathscr{A}_1(t)\mathscr{W}(x,y,z,t;\frac 32C'_1, \frac 32C''_1,\frac 32C''_1;A)\\
&+O(1)A^{-(1-\theta)\gamma}\left(N_0C_0\right) \mathscr{A}_3(t;2)\mathscr{W}(x,y,z,t;\frac 32C'_1, \frac 32C''_1,\frac 32 C''_1;A)\\
&+O(1)A^{-(1-\theta)\gamma}\left(N_0C_0\right)(1+t)^{-1}\left(e^{-t/2}+(1+t)^{-1/2}\right)
\left(1+A^2t^2\right)^{-1/2+\gamma/2}\mathscr{W}(x,y,z,t;\frac 32C'_1, \frac 32C''_1,\frac 32C''_1;A),
\end{aligned}
\end{equation}
with $\mathscr{A}_1(t)$ and $\mathscr{A}_3(t;\alpha)$ given by \eqref{aa1} and \eqref{aa3} respectively.

\vskip .05in

Combining \eqref{359}, \eqref{cin} and \eqref{361}, one obtains the pointwise estimate for $\nabla c(x,y,z,t)$.

It is similar to \eqref{cin} and \eqref{361} for the pointwise structure of $c(x,y,z,t)$: by Lemma \ref{lemma1}, Lemma \ref{lemma2} and Lemma \ref{cp1} with $\alpha=3/2$, one has that
\begin{equation}\nonumber
\begin{aligned}
\left|c(x,y,z,t)\right|
=&\left|\iiint_{\mathbb{R}^3}\mathbb{G}_{c,1}(x-x_0,y,z-z_0,t;y_0)c(x_0,y_0,z_0,0)dx_0dy_0dz_0\right.\\
&\left.+\int_0^t\iiint_{\mathbb{R}^3}
\mathbb{G}_{c,1}(x-x_0,y,z-z_0,t-\sigma;y_0)n(x_0,y_0,z_0,\sigma)dx_0dy_0dz_0d\sigma\right|\\
\le &CC_0 e^{-t}\mathscr{A}(t;A,\theta,\gamma) \mathscr{W}(x,y,t;C'_1,C''_1)+C\left(N_0C_0\right) \mathscr{A}_1(t)\mathscr{W}(x,y,z,t;\frac 32C'_1, \frac 32C''_1,\frac 32C''_1;A)\\
&+CA^{-(1-\theta)\gamma}\left(N_0C_0\right) \mathscr{A}_3(t;3/2)\mathscr{W}(x,y,z,t;\frac 32C'_1, \frac 32C''_1,\frac 32C''_1;A)\\
&+CA^{-(1-\theta)\gamma}\left(N_0C_0\right)(1+t)^{-1/2}\left(e^{-t/2}+(1+t)^{-1}\right)\left(1+A^2t^2\right)^{-1/2+\gamma/2}\mathscr{W}(x,y,z,t;\frac 32C'_1, \frac 32C''_1,\frac 32C''_1;A)\\
\le& CC_0 \mathscr{A}(t;A,\theta,\gamma)\mathscr{W}(x,y,z,t;\frac 32C'_1, \frac 32C''_1,\frac 32C''_1;A)
\cdot \left(A^{1/2-\gamma/2}\chi(t)\right).
\end{aligned}
\end{equation}
Here, $\chi(t)$ is a cut function defined by
\begin{equation}\label{chii}
\chi(t)\equiv \begin{cases}
0,\ \ \text{ for }\ 0<t\le 1\\
1,\ \ \text{ for }\ t>1\end{cases}
\end{equation}
and we finish the proof.
\end{proof}
Thus, Lemma \ref{cc1} together with \eqref{ansatz} results in that
\begin{equation}\label{none}
\big|\left(n\nabla c\right)(x,y,z,t)\big|
\le 8\left(N_0N_{c1}C_0^2\right)\cdot \Big(\mathscr{A}(t;A,\theta,\gamma)\Big)^2 \mathscr{W}(x,y,z,t;\frac 35 C'_1,\frac 35 C''_1,\frac 35 C''_1;A)
\cdot \left(A^{1/2-\gamma/2}\chi(t)\right).
\end{equation}
\textbf{Justification of ansatz assumption.} Substitute \eqref{none} into \eqref{pre} to verify the ansatz assumption \eqref{ansatz}: for $0< t\le A^{-\theta}$, Lemma \ref{lemma2} yields that
\begin{equation}\label{1}
\begin{aligned}
&\left|\int_0^t \iiint_{\mathbb{R}^3}\mathbb{G}(x-x_0,y,z-z_0,t-\sigma;y_0) \nabla \left(n\nabla c\right)(x_0,y_0,z_0)d x_0 dy_0dz_0\right|\\
=&\left|\int_0^t \iiint_{\mathbb{R}^3}\nabla \mathbb{G}(x-x_0,y,z-z_0,t-\sigma;y_0) \cdot \left(n\nabla c\right)(x_0,y_0,z_0)d x_0 dy_0dz_0\right|\\
=&O(1)\left(N_0N_{c1}C_0^2\right)\int_0^{t} \iiint_{\mathbb{R}^3}(t-\sigma)^{-2}\left(1+A^2 (t-\sigma)^2\right)^{-1/2}\\
&\cdot \mathbb{W}(x-x_0,y,z-z_0,t-\sigma;y_0;C_1,C_1,C_1;A)
\cdot \mathscr{W}(x_0,y_0,z_0,\sigma;\frac 35C'_1,\frac 35C''_1,\frac 35C''_1;A)dx_0dy_0dz_0d\sigma\\
\le& CA^{-\theta/2} \left(N_0N_{c1}C_0^2\right) \mathscr{W}(x,y,z,t;\frac 9{10}C'_1,\frac 9{10}C''_1,\frac 9{10}C''_1;A)
< N_0C_0 \mathscr{W}(x,y,z,t;C'_1,C''_1,C''_1;A).
\end{aligned}
\end{equation}
Here, we take $A$ sufficiently large to ensure $CA^{-\theta/2}N_{c1}C_0< 1$.

For $A^{-\theta}<t\le 1$, one could apply Lemma \ref{lemma2} and
Lemma \ref{lemma3} to obtain that
\begin{equation}\label{2}
\begin{aligned}
&\left|\int_0^t \iiint_{\mathbb{R}^3}\mathbb{G}(x-x_0,y,z-z_0,t-\sigma;y_0) \nabla \left(n\nabla c\right)(x_0,y_0,z_0)d x_0 dy_0dz_0\right|\\
=&O(1)\left(N_0N_{c1}C_0^2\right)\int_0^{A^{-\theta}} \iiint_{\mathbb{R}^3}(t-\sigma)^{-2}\left(1+A^2 (t-\sigma)^2\right)^{-1/2}\\
&\cdot \mathbb{W}(x-x_0,y,z-z_0,t-\sigma;y_0;C_1,C_1,C_1;A)
\cdot \mathscr{W}(x_0,y_0,z_0,\sigma;\frac 35C'_1,\frac 35C''_1,\frac 35C''_1;A)dx_0dy_0dz_0d\sigma\\
&+O(1)A^{-2(1-\theta)\gamma}\left(N_0N_{c1}C_0^2\right)\int_{A^{-\theta}}^t \iiint_{\mathbb{R}^3}(t-\sigma)^{-2}\left(1+A^2 (t-\sigma)^2\right)^{-1/2}\sigma^{-1/2-\gamma/2}
\left(1+A^2\sigma^2\right)^{-1/2+\gamma/2}
\\
&\cdot \mathbb{W}(x-x_0,y,z-z_0,t-\sigma;y_0;C_1,C_1,C_1;A)
\cdot \mathscr{W}(x_0,y_0,z_0,\sigma;\frac 35C'_1,\frac 35C''_1,\frac 35C''_1;A)dx_0dy_0dz_0d\sigma\\
\le& CA^{-\epsilon_0-(1-\theta)\gamma} \left(N_0N_{c1}C_0^2\right) t^{-1/4-\gamma/4}\left(1+A^2t^2\right)^{-1/4+\gamma/4} \mathscr{W}(x,y,z,t;\frac 9{10}C'_1,\frac 9{10}C''_1,\frac 9{10}C''_1;A)\\
&+CA^{-2(1-\theta)\gamma}\left(N_0N_{c1}C_0^2\right) t^{-1/4-\gamma/4}\left(1+A^2t^2\right)^{-1/4+\gamma/4}\cdot \mathscr{W}(x,y,z,t;\frac 9{10}C'_1,\frac 9{10}C''_1,\frac 9{10}C''_1;A)\\
\ll& A^{-(1-\theta)\gamma}\left(N_0C_0\right) t^{-1/4-\gamma/4}\left(1+A^2t^2\right)^{-1/4+\gamma/4} \mathscr{W}(x,y,z,t;C'_1,C''_1,C''_1;A).
\end{aligned}
\end{equation}
Here one also takes $A$ sufficiently large so that $CA^{-\epsilon_0}N_{c1}C_0+CA^{-(1-\theta)\gamma}N_{c1}C_0< 1$.

\vskip .05in

For $t>1$, Lemma \ref{lemma4} together with Lemma \ref{lemma2} and \ref{lemma3} results in that
\begin{multline}\label{3}
\begin{aligned}
&\left|\int_0^t \iiint_{\mathbb{R}^3}\mathbb{G}(x-x_0,y,z-z_0,t-\sigma;y_0) \nabla \left(n\nabla c\right)(x_0,y_0,z_0)d x_0 dy_0dz_0\right|\\
=&O(1)\left(N_0N_{c1}C_0^2\right)\int_0^{A^{-\theta}} \iiint_{\mathbb{R}^3}(t-\sigma)^{-2}\left(1+A^2 (t-\sigma)^2\right)^{-1/2}\\
&\cdot \mathbb{W}(x-x_0,y,z-z_0,t-\sigma;y_0;C_1,C_1,C_1;A)
\cdot \mathscr{W}(x_0,y_0,z_0,\sigma;\frac 35C'_1,\frac 35C''_1,\frac 35C''_1;A)dx_0dy_0dz_0d\sigma\\
&+O(1)A^{-2(1-\theta)\gamma}\left(N_0N_{c1}C_0^2\right)\int_{A^{-\theta}}^1 \iiint_{\mathbb{R}^3}(t-\sigma)^{-2}\left(1+A^2 (t-\sigma)^2\right)^{-1/2}\sigma^{-1/2-\gamma/2}\left(1+A^2\sigma^2\right)^{-1/2+\gamma/2}\\
&\cdot \mathbb{W}(x-x_0,y,z-z_0,t-\sigma;y_0;C_1,C_1,C_1;A)
\cdot \mathscr{W}(x_0,y_0,z_0,\sigma;\frac 35C'_1,\frac 35C''_1,\frac 35C''_1;A)dx_0dy_0dz_0d\sigma\\
&+O(1)A^{1/2-\gamma/2-2(1-\theta)\gamma}\left(N_0N_{c1}C_0^2\right)\int_{1}^t \iiint_{\mathbb{R}^3}(t-\sigma)^{-2}\left(1+A^2 (t-\sigma)^2\right)^{-1/2}
(1+\sigma)^{-3}\left(1+A^2\sigma^2\right)^{-1+\gamma}\\
&\cdot \mathbb{W}(x-x_0,y,z-z_0,t-\sigma;y_0;C_1,C_1,C_1;A)
\cdot \mathscr{W}(x_0,y_0,z_0,\sigma;\frac 35C'_1,\frac 35C''_1,\frac 35C''_1;A)dx_0dy_0dz_0d\sigma\\
\le& CA^{1-2\theta}\left(N_0N_{c1}C_0^2\right)(1+t)^{-2}\left(1+A^2t^2\right)^{-1/2}\cdot \mathscr{W}(x,y,z,t;\frac 9{10}C'_1,\frac 9{10}C''_1,\frac 9{10}C''_1;A)\\
&+CA^{-2(1-\theta)\gamma} A^{\gamma}
\left(N_0N_{c1}C_0^2\right)(1+t)^{-2}\left(1+A^2t^2\right)^{-1/2}\cdot \mathscr{W}(x,y,z,t;\frac 9{10}C'_1,\frac 9{10}C''_1,\frac 9{10}C''_1;A)\\
&+CA^{1/2-\gamma/2-2(1-\theta)\gamma}\left(N_0N_{c1}C_0^2\right)(1+t)^{-2}\left(1+A^2t^2\right)^{-1/2}\cdot \mathscr{W}(x,y,z,t;\frac 9{10}C'_1,\frac 9{10}C''_1,\frac 9{10}C''_1;A)\\
\le& A^{-(1-\theta)\gamma}\left(N_0C_0\right)(1+t)^{-2}\left(1+A^2t^2\right)^{-1/2+\gamma/2}\cdot \mathscr{W}(x,y,z,t;C'_1,C''_1,C''_1;A).
\end{aligned}
\end{multline}
The final inequality comes from
\begin{equation}\label{331}
CA^{1-2\theta}N_{c1}C_0\left(1+A^2t^2\right)^{-\gamma/2}\ll A^{-\gamma}\ll A^{-(1-\theta)\gamma},
\end{equation}
\begin{equation}\label{332}
CA^{-2(1-\theta)\gamma} A^{\gamma}N_{c1}C_0\left(1+A^2t^2\right)^{-\gamma/2}
\le CA^{-2(1-\theta)\gamma}N_{c1}C_0
\ll A^{-(1-\theta)\gamma},
\end{equation}
\begin{equation}\label{333}
CA^{1/2-\gamma/2-2(1-\theta)\gamma}N_{c1}C_0\left(1+A^2t^2\right)^{-\gamma/2}
\le CA^{1/2-3\gamma/2-2(1-\theta)\gamma}N_{c1}C_0
\ll A^{-(1-\theta)\gamma},\end{equation}
for $t>1$, $A\gg 1$, $\theta\in(2/3,1)$ and $\gamma\in(1/3,1/2]$.

\vskip .05in

Finally, one could combine \eqref{linn} and \eqref{1}-\eqref{3} to verify the ansatz assumption \eqref{ansatz} and finish the proof of Theorem \ref{main}.

\subsection{Proof of Theorem \ref{main0}: Parabolic-Elliptic Case}

For the parabolic-elliptic case, i.e. $\epsilon\equiv 0$ in \eqref{pro}, the solution $c(x,y,z,t)$ could be represented by \eqref{c0} and thus from Lemma \ref{cp2}, there exists positive constant $N_{c0}>N_0$ such that
\begin{equation}\label{nablac0}
\begin{aligned}
\left|\nabla c(x,y,z,t)\right|
=&\left|\iiint_{\mathbb{R}^3}\nabla\mathbb{G}_{c,0}(x-x_0,y-y_0,z-z_0)n(x_0,y_0,z_0)dx_0dy_0dz_0\right|\\
\le& C(N_0C_0)\mathscr{A}(t;A,\theta,\gamma)\iiint_{\mathbb{R}^3}\left(1+\frac{1}{\sqrt{(x-x_0)^2+(y-y_0)^2+(z-z_0)^2}} \right)\frac{e^{-\sqrt{(x-x_0)^2+(y-y_0)^2+(z-z_0)^2}}}{\sqrt{(x-x_0)^2+(y-y_0)^2+(z-z_0)^2}}\\ &\cdot\mathscr{W}(x_0,y_0,z_0,t;C'_1,C''_1,C''_1;A)dx_0dy_0dz_0\\
\le& N_{c0}C_0\mathscr{A}(t;A,\theta,\gamma) \mathscr{W}(x,y,z,t;\frac 32C'_1, \frac 32C''_1,\frac 32C''_1;A),
\end{aligned}
\end{equation}
\begin{equation}\label{318}
\begin{aligned}
\left|c(x,y,z,t)\right|
=&\left|\iiint_{\mathbb{R}^3}\mathbb{G}_{c,0}(x-x_0,y-y_0,z-z_0)n(x_0,y_0,z_0)dx_0dy_0dz_0\right|\\
\le& C(N_0C_0)\mathscr{A}(t;A,\theta,\gamma)\iiint_{\mathbb{R}^3}\frac{e^{-\sqrt{(x-x_0)^2+(y-y_0)^2+(z-z_0)^2}}}{\sqrt{(x-x_0)^2+(y-y_0)^2+(z-z_0)^2}}\cdot\mathscr{W}(x_0,y_0,z_0,t;C'_1,C''_1,C''_1)dx_0dy_0dz_0\\
\le& CC_0\mathscr{A}(t;A,\theta,\gamma) \mathscr{W}(x,y,z,t;\frac 32C'_1, \frac 32C''_1,\frac 32C''_1;A).
\end{aligned}
\end{equation}
Thus, based on the ansatz assumption \eqref{ansatz} for $n(x,y,z,t)$ and the pointwise structures \eqref{nablac0} for $\nabla c(x,y,z,t)$, the nonlinear term $n\nabla c$ in \eqref{pre} for parabolic-elliptic case could be estimated as follows:
$$
\left|n\nabla c(x,y,t)\right|
\le 8\left(N_0N_{c0}C_0^2\right) \Big( \mathscr{A}(t;A,\theta,\gamma)\Big)^2 \mathscr{W}(x,y,z,t;\frac 35 C'_1,\frac 35 C''_1,\frac 35 C''_1;A).
$$
One could substitute it into \eqref{pre} to justify the ansatz assumption \eqref{ansatz}: for $0< t\le A^{-\theta}$ and $A^{-\theta}<t\le 1$,
the justification is same as \eqref{1} and \eqref{2}. For $t>1$, it is a little different:
\begin{multline}\label{33}
\begin{aligned}
&\left|\int_0^t \iiint_{\mathbb{R}^3}\mathbb{G}(x-x_0,y,z-z_0,t-\sigma;y_0) \nabla \left(n\nabla c\right)(x_0,y_0,z_0)d x_0 dy_0dz_0\right|\\
=&O(1)\left(N_0N_{c0}C_0^2\right)\int_0^{A^{-\theta}} \iiint_{\mathbb{R}^3}(t-\sigma)^{-2}\left(1+A^2 (t-\sigma)^2\right)^{-1/2}\\
&\cdot \mathbb{W}(x-x_0,y,z-z_0,t-\sigma;y_0;C_1,C_1,C_1;A)
\cdot \mathscr{W}(x_0,y_0,z_0,\sigma;\frac 35C'_1,\frac 35C''_1,\frac 35C''_1;A)dx_0dy_0dz_0d\sigma\\
&+O(1)A^{-2(1-\theta)\gamma}\left(N_0N_{c0}C_0^2\right)\int_{A^{-\theta}}^1 \iiint_{\mathbb{R}^3}(t-\sigma)^{-2}\left(1+A^2 (t-\sigma)^2\right)^{-1/2}\sigma^{-1/2-\gamma/2}\left(1+A^2\sigma^2\right)^{-1/2+\gamma/2}\\
&\cdot \mathbb{W}(x-x_0,y,z-z_0,t-\sigma;y_0;C_1,C_1,C_1;A)
\cdot \mathscr{W}(x_0,y_0,z_0,\sigma;\frac 35C'_1,\frac 35C''_1,\frac 35C''_1;A)dx_0dy_0dz_0d\sigma\\
&+O(1)A^{-2(1-\theta)\gamma}\left(N_0N_{c0}C_0^2\right)\int_{1}^t \iiint_{\mathbb{R}^3}(t-\sigma)^{-2}\left(1+A^2 (t-\sigma)^2\right)^{-1/2}
(1+\sigma)^{-3}\left(1+A^2\sigma^2\right)^{-1+\gamma}
\\
&\cdot \mathbb{W}(x-x_0,y,z-z_0,t-\sigma;y_0;C_1,C_1,C_1;A)
\cdot \mathscr{W}(x_0,y_0,z_0,\sigma;\frac 35C'_1,\frac 35C''_1,\frac 35C''_1;A)dx_0dy_0dz_0d\sigma\\
\le& CA^{1-2\theta}\left(N_0N_{c0}C_0^2\right)(1+t)^{-2}\left(1+A^2t^2\right)^{-1/2}\cdot \mathscr{W}(x,y,z,t;\frac 9{10}C'_1,\frac 9{10}C''_1,\frac 9{10}C''_1;A)\\
&+CA^{-2(1-\theta)\gamma} A^{\gamma}
\left(N_0N_{c0}C_0^2\right)(1+t)^{-2}\left(1+A^2t^2\right)^{-1/2}\cdot \mathscr{W}(x,y,z,t;\frac 9{10}C'_1,\frac 9{10}C''_1,\frac 9{10}C''_1;A)\\
&+CA^{-2(1-\theta)\gamma}\left(N_0N_{c0}C_0^2\right)(1+t)^{-2}\left(1+A^2t^2\right)^{-1/2}\cdot \mathscr{W}(x,y,z,t;\frac 9{10}C'_1,\frac 9{10}C''_1,\frac 9{10}C''_1;A)\\
\ll& A^{-(1-\theta)\gamma}\left(N_0C_0\right)(1+t)^{-2}\left(1+A^2t^2\right)^{-1/2+\gamma/2}\cdot \mathscr{W}(x,y,z,t;C'_1,C''_1,C''_1;A).
\end{aligned}
\end{multline}
The final inequality comes from \eqref{331}, \eqref{332} and
\begin{equation}\label{334}
CA^{-2(1-\theta)\gamma}N_cC_0\left(1+A^2t^2\right)^{-\gamma/2}\ll A^{-(1-\theta)\gamma},
\end{equation}
for $t>1$, $A\gg 1$, $\theta\in(2/3,1)$ and $\gamma\in(0,1/2]$. Different from the parabolic-parabolic case, a non-zero lower bound $\frac 13$ is not needed for $\gamma$ here.

Finally, one could combine \eqref{linn}, \eqref{1}, \eqref{2} and \eqref{33} to verify the ansatz assumption \eqref{ansatz} and finish the proof of Theorem \ref{main0}.

\setcounter{equation}{0}

\section{Appendix}

In the Appendix, we list the tedious computations for the interactions between different wave patterns. First, we list all the wave pattern functions in this paper:
\begin{equation}\label{greenwave}
\mathbb{W}(x,y,z,t;y_0;D_1,D_2,D_3;A)\equiv \exp{\left(\frac{-\left(x-\frac{At}{2}(y+y_0)\right)^2}{4D_1t\left(1+\frac{1}{12}A^2 t^2\right)}-\frac{(y-y_0)^2}{4D_2t}-\frac{z^2}{4D_3t}\right)},
\end{equation}

\begin{equation}\label{wavexx}
\mathscr{W}^x(x,y,t;D_1,D_2;A)\equiv \exp{\left(\frac{-\left(x-\frac{At}2 y\right)^2}{D_1t\left(1+A^2 t^2\right)}\right)}+\exp{\left(-\frac{\left|x-\frac{A t}2 y\right|}{D_2(1+At)}\right)},
\end{equation}

\begin{equation}\label{waveyz}
\mathscr{W}^o(y,t;D_1,D_2)\equiv \exp{\left(-\frac{y^2}{D_1t}\right)+\exp{\left(-\frac{|y|}{D_2}\right)}},
\end{equation}

\begin{multline}\label{wavest}
\begin{aligned}
\mathscr{W}(x,y,z,t;C'_1,C''_1,C''_1;A)\equiv&\left(\exp{\left(\frac{-\left(x-\frac{At}2 y\right)^2}{C'_1t\left(1+A^2 t^2\right)}\right)}+\exp{\left(-\frac{\left|x-\frac{A t}2 y\right|}{C'_1(1+At)}\right)}\right)\\
&\cdot\left(\exp{\left(-\frac{y^2}{C''_1t}\right)}+\exp{\left(-\frac{|y|}{C''_1}\right)}\right)
\cdot\left(\exp{\left(-\frac{z^2}{C''_1t}\right)}+\exp{\left(-\frac{|z|}{C''_1}\right)}\right),
\end{aligned}
\end{multline}
and one has that
\begin{lemma}\label{lemmaz} For a positive constant $C_1>1$ and $\sigma\in(0,t)$, there exists positive constants $C''>90 C_1$ such that
\begin{multline}\label{nnwzz}
\begin{aligned}
&\int_{\mathbb{R}}
\mathbb{W}(x-x_0,y,z-z_0,t-\sigma;y_0;C_1,C_1,C_1;A)
\cdot \mathscr{W}^o(z_0,\sigma;C'',C'')dz_0\\
=&O(1)\left(\min\left\{(t-\sigma)^{1/2},\sigma^{1/2}\right\}+\min\left\{(t-\sigma)^{1/2},1\right\}\right)
\cdot \mathbb{W}^{x,y}(x-x_0,y,t-\sigma;y_0;C_1,C_1;A)\cdot \mathscr{W}^o(z,t;\frac 32 C'',\frac 32 C'').
\end{aligned}
\end{multline}
Here, one denotes
\begin{equation}\label{wxyxy}
\mathbb{W}^{x,y}(x,y,t;y_0;D_1,D_2;A)\equiv \exp{\left(\frac{-\left(x-\frac{At}{2}(y+y_0)\right)^2}{4D_1t\left(1+\frac{1}{12}A^2 t^2\right)}-\frac{(y-y_0)^2}{4D_2t}\right)},
\end{equation}
the functions $\mathbb{W}(x,y,z,t;y_0;D_1,D_2,D_3;A)$ and $\mathscr{W}^o(z,t;D_1,D_2)$ are defined by \eqref{greenwave} and \eqref{waveyz} respectively.
\end{lemma}
\begin{proof}
There exists a constant $C>1$ such that
\begin{multline}\nonumber
\begin{aligned}
&\exp{\left(-\frac{(z-z_0)^2}{4C_1(t-\sigma)}\right)}\cdot \mathscr{W}^o(z_0,\sigma;C'',C'')\\
\le& C\left(\exp{\left(-\frac{z^2}{128C_1t}\right)}+\mathscr{W}^o(z,t;\frac 32C'',\frac 32C'')\right)\cdot\exp{\left(-\frac{(z-z_0)^2}{CC_1(t-\sigma)}\right)}\cdot \mathscr{W}^o(z_0,\sigma;CC'',CC'').
\end{aligned}
\end{multline}
Thus
\begin{multline}\nonumber
\begin{aligned}
&\int_{\mathbb R}\exp{\left(-\frac{(z-z_0)^2}{4C_1(t-\sigma)}\right)}\cdot \mathscr{W}^o(z_0,\sigma;C'',C'') dz_0\\
\le&
C\left(\exp{\left(-\frac{z^2}{128C_1t}\right)}+\mathscr{W}^o(z,t;\frac 32C'',\frac 32C'')\right)\cdot\int_{\mathbb R} \exp{\left(-\frac{(z-z_0)^2}{CC_1(t-\sigma)}\right)}dz_0\\
\le&
C(t-\sigma)^{1/2}\mathscr{W}^o(z,t;\frac 32C'',\frac 32C''),
\end{aligned}
\end{multline}
and also similarly
\begin{multline}\nonumber
\begin{aligned}
&\int_{\mathbb R}\exp{\left(-\frac{(z-z_0)^2}{4C_1(t-\sigma)}\right)}\cdot \mathscr{W}^o(z_0,\sigma;C'',C'') dz_0\\
\le&
C\left(\exp{\left(-\frac{z^2}{128C_1t}\right)}+\mathscr{W}^o(z,t;\frac 32C'',\frac 32C'')\right)\cdot\int_{\mathbb R}\mathscr{W}^o(z_0,\sigma;CC'',CC'')dz_0\\
\le&
C(\sigma^{1/2}+1)\mathscr{W}^o(z,t;\frac 32C'',\frac 32C''),
\end{aligned}
\end{multline}
and we finish the proof.
\end{proof}

\begin{lemma}\label{lemma0} For a positive constant $C_1>1$ and $\sigma\in(0,t)$, there exist positive constants $C',C''$ and $1<B\le\frac{3}2$ such that
\begin{multline}\label{nnw}
\begin{aligned}
&\iint_{\mathbb{R}^3}
\mathbb{W}^{xy}(x-x_0,y,t-\sigma;y_0;C_1,C_1;A)
\cdot \exp\left({\frac{-\left(x_0-\frac{A\sigma}2 y_0\right)^2}{C'\sigma\left(1+A^2 \sigma^2\right)}}-\frac{y_0^2}{C''\sigma}\right)dx_0dy_0\\
=&O(1)\min\left\{(t-\sigma)\left(1+A^2(t-\sigma)^2\right)^{1/2},\sigma\left(1+A^2\sigma^2\right)^{1/2}\right\}
\cdot\exp\left({\frac{-\left(x-\frac{At}2 y\right)^2}{BC' t\left(1+A^2 t^2\right)}}-\frac{y^2}{BC'' t}\right),
\end{aligned}
\end{multline}

\begin{multline}\label{nnw1}
\begin{aligned}
&\iint_{\mathbb{R}^2}
\mathbb{W}^{xy}(x-x_0,y,t-\sigma;y_0;C_1,C_1;A)
\cdot \exp{\left(-\frac{\left|x_0-\frac{A \sigma}2 y_0\right|}{C'(1+A\sigma)}-\frac{\left|y_0\right|}{C''}\right)}dx_0dy_0\\
=&O(1)\min\left\{(t-\sigma)\left(1+A^2(t-\sigma)^2\right)^{1/2},(t-\sigma)^{1/2}\left(1+A^2(t-\sigma)^2\right)^{1/2},(t-\sigma)^{1/2}\left(1+A\sigma\right), 1+A\sigma\right\}\\
&\cdot \mathscr{W}^x(x,y,t;BC',BC';A)\mathscr{W}^o(y,t;BC'',BC''),
\end{aligned}
\end{multline}

\begin{multline}\label{nnw2}
\begin{aligned}
&\iint_{\mathbb{R}^2}
\mathbb{W}^{xy}(x-x_0,y,t-\sigma;y_0;C_1,C_1;A)
\cdot \exp{\left(\frac{-\left(x_0-\frac{A\sigma}2 y_0\right)^2}{C'\sigma\left(1+A^2 \sigma^2\right)}-\frac{\left|y_0\right|}{C''}\right)}dx_0dy_0\\
=&O(1)\min\left\{(t-\sigma)\left(1+A^2(t-\sigma)^2\right)^{1/2},(t-\sigma)^{1/2}\left(1+A^2(t-\sigma)^2\right)^{1/2},(t-\sigma)^{1/2}\sigma^{1/2}\left(1+A^2\sigma^2\right)^{1/2}, \right.\\
&\left.\sigma^{1/2}\left(1+A^2\sigma^2\right)^{1/2}\right\}
\cdot \mathscr{W}^x(x,y,t;BC',BC';A)\mathscr{W}^o(y,t;BC'',BC''),
\end{aligned}
\end{multline}

\begin{multline}\label{nnw3}
\begin{aligned}
&\iint_{\mathbb{R}^2}
\mathbb{W}^{xy}(x-x_0,y,t-\sigma;y_0;C_1,C_1;A)
\cdot \exp{\left(-\frac{\left|x_0-\frac{A \sigma}2 y_0\right|}{C'(1+A\sigma)}-\frac{y_0^2}{C''\sigma}\right)}dx_0dy_0\\
=&O(1)\min\left\{(t-\sigma)\left(1+A^2(t-\sigma)^2\right)^{1/2},\sigma^{1/2}(t-\sigma)^{1/2}\left(1+A^2(t-\sigma)^2\right)^{1/2},(t-\sigma)^{1/2}\left(1+A\sigma\right), \sigma^{1/2}\left(1+A\sigma\right)\right\}\\
&\cdot \mathscr{W}^x(x,y,t;BC',BC';A)\cdot \exp\left(-\frac{y^2}{BC''t}\right).
\end{aligned}
\end{multline}
\end{lemma}

\begin{proof}
For and fixed $(x,y,t)$ and any fixed $y_0\in \mathbb{R}$ and $0\le\sigma\le t$, consider the region
$$
\left\{x_0\in\mathbb{R}: \left|x-\frac{At}2 y-x_0+\frac{A\sigma}2 y_0\right|\ge\frac{\theta_1A\sigma}2\left|y-y_0\right|+\frac{\theta_1A(t-\sigma)}2\left|y_0\right|\right\}\cap\left\{x_0\in\mathbb{R}: \left|x-\frac{At}2y\right|\ge \theta_2\left|x_0-\frac{At}2y_0\right|\right\}
$$
for $\theta_1, \theta_2>1$ and one has that
\begin{multline}\label{220}
\begin{aligned}
\left|x-x_0-\frac{A(t-\sigma)}{2}(y+y_0)\right|=&\left|x-\frac{At}2 y-x_0+\frac{A\sigma}2 y_0+\frac{A\sigma}2\left(y-y_0\right)-\frac{A(t-\sigma)}2y_0\right|\\
\ge& \frac{\theta_1-1}{\theta_1} \left|x-\frac{At}2 y-x_0+\frac{A\sigma}2 y_0\right|
\ge \frac{\theta_1-1}{\theta_1}\cdot\frac{\theta_2-1}{\theta_2} \left|x-\frac{At}2 y\right|.
\end{aligned}
\end{multline}
In the region
$$
\left\{x_0\in\mathbb{R}: \left|x-\frac{At}2 y-x_0+\frac{A\sigma}2 y_0\right|\le\frac{\theta_1A\sigma}2\left|y-y_0\right|+\frac{\theta_1A(t-\sigma)}2\left|y_0\right|\right\}\cap\left\{x_0\in\mathbb{R}: \left|x-\frac{At}2y\right|\ge \theta_2\left|x_0-\frac{At}2y_0\right|\right\},
$$
it holds that
\begin{multline}\label{221}
\begin{aligned}
\frac{\left(y-y_0\right)^2+y_0^2}{t}\ge& \frac{\left(At\left|y-y_0\right|+At\left|y_0\right|\right)^2}{2A^2t^3}\ge
\frac{\left(A\sigma\left|y-y_0\right|+A(t-\sigma)\left|y_0\right|\right)^2}{2A^2t^3}\\
\ge& \frac{2}{\theta_1^2}\frac{\left|x-\frac{At}2 y-x_0+\frac{A\sigma}2 y_0\right|^2}{A^2t^3}
\ge \frac{2}{\theta_1^2}\left(\frac{\theta_2-1}{\theta_2}\right)^2 \frac{\left|x-\frac{At}2 y\right|^2}{t\left(1+A^2t^2\right)}.
\end{aligned}
\end{multline}
And if $x_0$ satisfies $\left|x-\frac{At}2y\right|\le \theta_2\left|x_0-\frac{At}2y_0\right|$, it holds that
\begin{equation}\label{222}
\left|x_0-\frac{A\sigma}2y_0\right|^2\ge \frac{1}{\theta_2^2}\left|x-\frac{At}2y\right|^2.
\end{equation}
Combining \eqref{220}-\eqref{222}, one has that
\begin{equation}
\begin{aligned}
&\iint_{\mathbb{R}^2}
\mathbb{W}^{xy}(x-x_0,y,t-\sigma;y_0;C_1,C_1;A)\cdot \exp\left({\frac{-\left(x_0-\frac{A\sigma}2 y_0\right)^2}{C'\sigma\left(1+A^2 \sigma^2\right)}}-\frac{y_0^2}{C''\sigma}\right)dx_0dy_0\\
\le& C\iint_{\mathbb{R}^2}
\left(\exp\left({\frac{-\left(\frac{\theta_1-1}{\theta_1}\cdot\frac{\theta_2-1}{\theta_2}\right)^2\left(x-\frac{At}2 y\right)^2}{4C_1 t\left(1+A^2 t^2\right)}}\right)+\exp\left({\frac{-\frac{2}{\theta_1^2}\left(\frac{\theta_2-1}{\theta_2}\right)^2\left(x-\frac{At}2 y\right)^2}{18C'' t\left(1+A^2 t^2\right)}}\right)+\exp\left({\frac{-\frac{1}{\theta_2^2}\left(x-\frac{At}2 y\right)^2}{\frac 98C' t\left(1+A^2 t^2\right)}}\right)\right) \\
&\cdot\exp{\left(-\frac{(y-y_0)^2}{8C_1(t-\sigma)}\right)}\cdot \exp\left({\frac{-\left(x_0-\frac{A\sigma}2 y_0\right)^2}{9C'\sigma\left(1+A^2 \sigma^2\right)}}-\frac{y_0^2}{\frac {18}{17}C''\sigma}\right)dx_0dy_0\\
\le& C\exp\left({\frac{-\left(x-\frac{At}2 y\right)^2}{BC' t\left(1+A^2 t^2\right)}}\right)\iint_{\mathbb{R}^2}\exp{\left(-\frac{(y-y_0)^2}{8C_1(t-\sigma)}\right)}\cdot \exp\left({\frac{-\left(x_0-\frac{A\sigma}2 y_0\right)^2}{9C'\sigma\left(1+A^2 \sigma^2\right)}}-\frac{y_0^2}{\frac {18}{17}C''\sigma}\right)dx_0dy_0\\
\le& C\sigma^{1/2}\left(1+A^2\sigma^2\right)^{1/2} \exp\left({\frac{-\left(x-\frac{At}2 y\right)^2}{BC' t\left(1+A^2 t^2\right)}}\right)\\
&\cdot \left(\int_{\left|y\right|\ge\theta_3\left|y_0\right|}+\int_{\left|y\right|\le\theta_3\left|y_0\right|}\right)\exp{\left(-\frac{(y-y_0)^2}{8C_1(t-\sigma)}-\frac{y_0^2}{\frac 98 C''\sigma}\right)}\cdot \exp{\left(-\frac{y_0^2}{{18}C''\sigma}\right)}dy_0\\
\le& C\sigma\left(1+A^2\sigma^2\right)^{1/2} \exp\left({\frac{-\left(x-\frac{At}2 y\right)^2}{BC' t\left(1+A^2 t^2\right)}}-\frac{y^2}{BC''t}\right)
\end{aligned}
\end{equation}
in which we choose $1< \theta_2^2, \theta_3^2\le\frac 43$ and the other constants $\theta_1, C', C''$ accordingly:
\begin{equation}\label{constant1}
4C_1\left(\frac{\theta_1}{\theta_1-1}\right)^2\left(\frac{\theta_2}{\theta_2-1}\right)^2\le\frac 32 C',\ \ \ \ \ \ \ \ 9C''{\theta_1^2}\left(\frac{\theta_2}{\theta_2-1}\right)^2\le\frac 32 C',\ \ \ \ \ \ \ \ 8C_1\left(\frac{\theta_3}{\theta_3-1}\right)^2\le \frac 32 C''
\end{equation}
and thus $1<B\le \frac32$. Here, in the computation of the integrals, we integrate $\exp\left({\frac{-\left(x_0-\frac{A\sigma}2 y_0\right)^2}{9C'\sigma\left(1+A^2 \sigma^2\right)}}\right)$ with respect to $x_0$ and $\exp{\left(-\frac{|y_0|^2}{\frac 1{18}C''\sigma}\right)}$ with respect to $y_0$ and obtain the factor $\sigma\left(1+A^2\sigma^2\right)^{1/2}$ after integration. One could similarly integrate
$$
\exp{\left(\frac{-\left(x-x_0-\frac{A(t-\sigma)}{2}(y+y_0)\right)^2}{4C_1(t-\sigma)\left(1+\frac{1}{12}A^2 (t-\sigma)^2\right)}-\frac{(y-y_0)^2}{4C_1(t-\sigma)}\right)}
$$
with respect to $x_0$ and $y_0$ respectively to obtain a factor $(t-\sigma)\left(1+A^2(t-\sigma)^2\right)^{1/2}$ from integration and finish the proof of \eqref{nnw}.

\vskip .05in

Next, one verifies \eqref{nnw1}: similar to \eqref{220}-\eqref{222}, one has that
\begin{equation}\nonumber
\begin{aligned}
&\mathbb{W}^{xy}(x-x_0,y,t-\sigma;y_0;C_1,2C_1)
\cdot \exp{\left(-\frac{\left|x_0-\frac{A \sigma}2 y_0\right|}{\frac 54C'(1+A\sigma)}-\frac{\left|y_0\right|}{10C''}\right)}\\
\le &
\exp\left({\frac{-\left(\frac{\theta_1-1}{\theta_1}\cdot\frac{\theta_2-1}{\theta_2}\right)^2\left(x-\frac{At}2 y\right)^2}{4C_1 t\left(1+A^2 t^2\right)}}\right)
+\exp\left({\frac{-\frac{1}{\theta_1^2}\left(\frac{\theta_2-1}{\theta_2}\right)^2\left(x-\frac{At}2 y\right)^2}{8C_1 t\left(1+A^2 t^2\right)}}\right)+\exp\left(-\frac{\frac{\theta_2-1}{\theta_1\theta_2}\left|x-\frac{A \sigma}2 y\right|}{10C''At}\right)+\exp\left(-\frac{\left|x-\frac{A \sigma}2 y\right|}{\frac 54\theta_2C'\left(1+At\right)}\right)
\end{aligned}
\end{equation}
and thus
\begin{multline}\nonumber
\begin{aligned}
&\iint_{\mathbb{R}^2}
\mathbb{W}^{xy}(x-x_0,y,t-\sigma;y_0;C_1,C_1)
\cdot \exp{\left(-\frac{\left|x_0-\frac{A \sigma}2 y_0\right|}{C'(1+A\sigma)}-\frac{\left|y_0\right|}{C''}\right)}dx_0dy_0\\
\le& C\left(\exp\left({\frac{-\left(x-\frac{At}2 y\right)^2}{BC' t\left(1+A^2 t^2\right)}}\right)+\exp\left(-\frac{\left|x-\frac{A \sigma}2 y\right|}{BC'\left(1+At\right)}\right)\right)
\cdot\iint_{\mathbb{R}^2} \exp{\left(-\frac{(y-y_0)^2}{8C_1(t-\sigma)}\right)}\cdot \exp{\left(-\frac{\left|x_0-\frac{A \sigma}2 y_0\right|}{5C'(1+A\sigma)}-\frac{\left|y_0\right|}{\frac{10}{9}C''}\right)}dx_0dy_0\\
\le& C\left(1+A\sigma\right)\left(\exp\left({\frac{-\left(x-\frac{At}2 y\right)^2}{BC' t\left(1+A^2 t^2\right)}}\right)+\exp\left(-\frac{\left|x-\frac{A \sigma}2 y\right|}{BC'\left(1+At\right)}\right)\right)\\
&\cdot \left(\int_{\left|y\right|\ge\theta_3\left|y_0\right|}+\int_{\left|y\right|\le\theta_3\left|y_0\right|}\right)\exp{\left(-\frac{(y-y_0)^2}{8C_1(t-\sigma)}-\frac{\left|y_0\right|}{\frac{5}{4}C''}\right)}\cdot \exp{\left(-\frac{\left|y_0\right|}{{10}C''}\right)}dy_0\\[2mm]
\le& C\mathscr{W}^x(x,y,t;BC',BC';A)\mathscr{W}^o(y,t;BC'',BC'').
\end{aligned}
\end{multline}
Here, one chooses $1<\theta_2, \theta_3\le \frac 65$ and the other constants $\theta_1, C', C''$ accordingly:
\begin{equation}\label{constant2}
4C_1\left(\frac{\theta_1}{\theta_1-1}\right)^2\left(\frac{\theta_2}{\theta_2-1}\right)^2\le\frac 32 C',\ \ \ \ \ \ \ \ 8C_1{\theta_1^2}\left(\frac{\theta_2}{\theta_2-1}\right)^2\le\frac 32 C',\ \ \ \ \ \ \ \ 10C''\frac{\theta_1\theta_2}{\theta_2-1}\le\frac 32 C',\ \ \ \ \ \ \ \ 8C_1\left(\frac{\theta_3}{\theta_3-1}\right)^2\le \frac 32 C''
\end{equation}
to ensure that $1<B\le \frac 32$. Here, we use $\exp\left(-\frac{\left|x_0-\frac{A \sigma}2 y_0\right|}{5C'(1+A\sigma)}-\frac{\left|y_0\right|}{\frac{10}9C''}\right)$ to integrate with respect to $x_0$ and $y_0$ to obtain the factor $1+A\sigma$ and one could also use the heat kernels to integrate to obtain other factors in \eqref{nnw1}.

The proofs of \eqref{nnw2} and \eqref{nnw3} are similar and we omit the details.
\end{proof}

\section*{Acknowledgement}
S. Deng is supported by National Nature Science Foundation of
China 11831011 and 12161141004 and Shanghai Science and Technology Innovation Action Plan No. 21JC1403600.
B. Shi is supported by the Jiangsu Funding Program for Excellent Postdoctoral Talent 2023ZB116. W. Wang is supported by National Nature Science Foundation of
China 12271357, 11831011 and 12161141004 and Shanghai Science and Technology Innovation Action Plan No. 21JC1403600.


\begin{thebibliography}{99}

\bibitem{agm} Y. Ascasibar, R. Granero-Belinchon and J. M. Moreno, An approximate treatment of gravitational collapse, Phys. D, 262 (2013), 71-82.

\bibitem{bh} J. Bedrossian and S. He, Suppression of blow-up in Patlak-Keller-Segel via shear flows, SIAM J. Math. Anal., 49 (2017), no. 6, 4722-4766.

\bibitem{cc} V. Calvez and L. Corrias, The parabolic-parabolic Keller-Segel model in $\mathbb{R}^2$, Commun. Math. Sci., 6 (2008), no. 2, 417-447.

\bibitem{ccdl} J. Che, L. Chen, B. Duan and Z. Luo, On the existence of local strong solutions to chemotaxis-shallow water system with large data and vacuum, J. Differential Equations, 261 (2016), no. 12, 6758-6789.

\bibitem{clm} T. Cieslak, P. Laurencot and C. Morales-Rodrigo, Global existence and convergence to steady states in a chemorepulsion system, In \textit{Parabolic and Navier-Stokes equations, Part 1}, volume 81 of Banach Center Publ., pages 105-117, Polish Acad. Sci. Inst. Math., Warsaw, 2008.

\bibitem{ckrz} P. Constantin, A. Kiselev, L. Ryzhik and A. Zlatos, Diffusion and mixing in fluid flow, Ann. of Math. (2), 168 (2008), no. 2, 643-674.

\bibitem{zg} M. Coti Zelati and T. Gallay, Enhanced dissipation and Taylor dispersion in higher-dimensional parallel shear flows, J. Lond. Math. Soc., 108 (2023), 1358-1392.

\bibitem{dwy} S. Deng, W. Wang and S.-H. Yu, Viscous conservation laws with boundary, SIAM J. Math. Anal., 44 (2012), no. 4, 2695-2755.

\bibitem{fsw} Y. Feng, B. Shi and W. Wang, Dissipation enhancement of planar helical flows and applications to three-dimensional Kuramoto-Sivashinsky and Keller-Segel equations, J. Differential Equations, 313 (2022), 420-449.

\bibitem{h} S. He, Suppression of blow-up in parabolic-parabolic Patlak-Keller-Segel via strictly monotone shear flows, Nonlinearity, 31 (2018), 3651-3688.

\bibitem{h1} S. He, Enhanced dissipation and blow-up suppression in a chemotaxis-fluid system, SIAM J. Math. Anal., 55 (2023), no. 4, 3651-3688.

\bibitem{hv} M. A. Herrero and J. J. L. Velazquez, Singularity patterns in a chemotaxis model, Math. Ann., 306 (1996), no. 3, 583-623.

\bibitem{hv2} M. A. Herrero and J. J. L. Velazquez, A blow-up mechanism for a chemotaxis model, Ann. Scuola Norm. Sup. Pisa Cl. Sci. (4), 24 (1997), no. 4, 633-683.

\bibitem{kx} A. Kiselev and X. Xu, Suppression of chemotactic explosion by mixing, Arch. Ration. Mech. Anal., 222 (2016), no. 2, 1077-1112.

\bibitem{ll} J.-G. Liu and A. Lorz, A coupled chemotaxis-fluid model: global existence, Ann. Inst. H. Poincare Anal. Non Lineaire, 28 (2011), no. 5, 643-652.

\bibitem{l} T.-P. Liu, Pointwise convergence to shock waves for viscous conservation laws, Comm. Pure Appl. Math., 50 (1997), no. 11, 1113-1182.

\bibitem{lw} T.-P. Liu and H. Wang, Viscous scalar rarefaction waves, SIAM J. Math. Anal., 49 (2017), 2061-2100.

\bibitem{ly} T.-P. Liu and S.-H. Yu, Viscous rarefaction waves, Bull. Inst. Math. Acad. Sin. (NS), 5 (2010), 123-179.

\bibitem{lz} T.-P. Liu and Y. Zeng, Time-asymptotic behavior of wave propagation around a viscous shock profile, Comm. Math. Phys., 290 (2009), no. 1, 23-82.

\bibitem{lz2} T.-P. Liu and Y. Zeng, Shock waves in conservation laws with physical viscosity, Mem. Amer. Math. Soc., 234 (2015).

\bibitem{lo} A. Lorz, Coupled chemotaxis fluid model, Math. Models Methods Appl. Sci., 20 (2010), no. 6, 987-1004.

\bibitem{mp} P. S. Marcus and W. H. Press, On Green's functions for small disturbances of plane Couette flow, J. Fluid Mech., 79 (1977), 525–534.

\bibitem{n} T. Nagai, Blow-up of radially symmetric solutions to a chemotaxis system, Adv. Math. Sci. Appl., 5 (1995), no. 2, 581-601.

\bibitem{s} R. Schweyer, Stable blow-up dynamic for the parabolic-parabolic Patlak-Keller-Segel model, arXiv: 1403.4975.

\bibitem{ss} T. Senba and T. Suzuki, Parabolic system of chemotaxis: blowup in a finite and the infinite time, 8 (2001), 349-367, IMS Workshop on Reaction-Diffusion Systems (Shatin, 1999).

\bibitem{sw} P. Souplet and M. Winkler, Blow-up profiles for the parabolic-elliptic Keller Segel system in dimensions $n\ge 3$, Comm. Math. Phys., 367 (2019), no. 2, 665-681.

\bibitem{w0} M. Winkler, Global large-data solutions in a chemotaxis-(Navier-)Stokes system modeling cellular swimming in fluid drops, Comm. Partial Differential Equations, 37 (2012), no. 2, 319-351.

\bibitem{w} M. Winkler, Finite-time blow-up in the higher-dimensional parabolic-parabolic Keller-Segel system, J. Math. Pures Appl. (9), 100 (2013), no. 5, 748-767.

\bibitem{y} S.-H. Yu, Nonlinear wave propagations over a Boltzmann shock profile, Journal of the American Mathematical Society, 23 (2010), no. 4, 1041-1118.

\bibitem{zzz} L. Zeng, Z. Zhang and R. Zi, Suppression of blow-up in Patlak-Keller-Segel-Navier-Stokes system via the Couette flow, J. Funct. Anal., 280 (2021), no. 10, Paper No. 108967.

\end{thebibliography}
\end{document}